\documentclass[a4paper,12pt]{article}

\usepackage{amsmath}
\usepackage{amsthm}
\usepackage{amssymb}

\usepackage{latexsym}
\usepackage{array}
\usepackage{enumerate, graphicx}

\numberwithin{equation}{section}

\newcommand{\rhom}{R{\mathcal{H}}om}
\newcommand{\BDC}{{\mathbf{D}}^{\mathrm{b}}}

\newcommand{\shom}{{\mathcal{H}}om}

\newcommand{\CC}{\mathbb{C}}
\newcommand{\NN}{\mathbb{N}}
\newcommand{\RR}{\mathbb{R}}
\newcommand{\QQ}{\mathbb{Q}}
\newcommand{\ZZ}{\mathbb{Z}}
\newcommand{\D}{\mathcal{D}}
\newcommand{\E}{\mathcal{E}}

\newcommand{\LL}{\mathcal{L}}
\newcommand{\M}{\mathcal{M}}
\newcommand{\N}{\mathcal{N}}
\newcommand{\sho}{\mathcal{O}}
\newcommand{\R}{\mathcal{R}}
\newcommand{\CS}{\mathcal{S}}
\newcommand{\CH}{\mathcal{H}}
\newcommand{\CK}{\mathcal{K}}
\newcommand{\CAC}{\mathcal{C}}
\newcommand{\CP}{\mathcal{P}}

\newcommand{\PP}{{\mathbb P}}

\newcommand{\Lin}{{\rm Lin}}
\newcommand{\Ker}{{\rm Ker}}
\newcommand{\cl}{{\rm cl}}
\newcommand{\GFT}{{ \wedge }}
\newcommand{\FT}{{ \vee }}
\newcommand{\rd}{{\rm rd}}
\newcommand{\an}{{\rm an}}

\renewcommand{\dim}{{\rm dim}}

\newcommand{\Vol}{{\rm Vol}}

\newcommand{\Eu}{{\rm Eu}}
\newcommand{\e}{\varepsilon}

\newcommand{\Spec}{{\rm Spec}}

\newcommand{\id}{{\rm id}}

\newcommand{\Int}{{\rm Int}}

\newcommand{\tl}[1]{\widetilde{#1}}
\newcommand{\wht}[1]{\widehat{#1}}

\newcommand{\simto}{\overset{\sim}{\longrightarrow}}

\newtheorem{definition}{Definition}[section]
\newtheorem{theorem}[definition]{Theorem}
\newtheorem{proposition}[definition]{Proposition}
\newtheorem{lemma}[definition]{Lemma}
\newtheorem{corollary}[definition]{Corollary}
\newtheorem{example}[definition]{Example}
\newtheorem{remark}[definition]{Remark}

\title{Confluent $A$-hypergeometric 
functions and rapid decay homology cycles 
}

\author{Alexander ESTEROV 
\footnote{ National Research University 
Higher School of Economics. \newline Faculty 
of Mathematics NRU HSE, 7 Vavilova 117312 
Moscow, Russia. 
E-mail: aesterov@hse.ru. \newline 
This study (research grant No 14-01-0152) 
was supported by The National 
Research University -- Higher School 
of Economics Academic 
Fund Program in 2014/2015. Partially 
supported by RFBR 
grants 13-01-00755 and 12-01-31233, 
MESRF grant MK-6223.2012.1, 
and the Dynasty Foundation fellowship.} 
and Kiyoshi TAKEUCHI 
\footnote{Institute of Mathematics, University  of 
Tsukuba, 1-1-1, Tennodai, 
Tsukuba, Ibaraki, 305-8571, Japan. 
E-mail: takemicro@nifty.com } }

\date{}

\begin{document}

\maketitle

\begin{abstract}
We study confluent $A$-hypergeometric functions 
introduced by Adolphson  
\cite{A}. In particular, we give 
their integral representations 
by using rapid decay homology 
cycles of Hien \cite{H-1} and 
\cite{H-2}. The method of toric 
compactifications introduced in \cite{L-S} 
and \cite{M-T-3} will be used to prove 
our main theorem. Moreover we apply it to 
obtain a formula for the asymptotic expansions 
at infinity of confluent 
$A$-hypergeometric functions. 
\end{abstract}

\section{Introduction}\label{sec:1}

The theory of $A$-hypergeometric systems 
introduced by Gelfand-Graev-Kapranov-Zelevinsky 
\cite{G-G-Z}, \cite{G-K-Z-1} is a vast generalization 
of that of classical hypergeometric differential 
equations. We call their holomorphic solutions 
$A$-hypergeometric functions. 
As in the case of classical hypergeometric 
functions, $A$-hypergeometric functions 
admit $\Gamma$-series expansions 
(\cite{G-K-Z-1}) and integral 
representations (\cite{G-K-Z-2}). Moreover 
this theory has deep connections with 
many other fields of mathematics, such as 
toric varieties, commutative algebra, 
projective duality, 
period integrals, mirror symmetry, 
enumerative algebraic geometry and 
combinatorics. Also from the viewpoint of 
the $\D$-module theory (see  
\cite{H-T-T}), 
$A$-hypergeometric 
$\D$-modules are very 
elegantly constructed 
in \cite{G-K-Z-2}. For the recent 
development of this subject see \cite{S-S-T} and 
\cite{S-W}. In \cite{B}, \cite{B-H}, \cite{G-K-Z-2}, 
\cite{Horja} and \cite{T-1} the 
monodromies of their $A$-hypergeometric functions 
were studied. Recall that in the theory of 
Gelfand-Graev-Kapranov-Zelevinsky 
\cite{G-G-Z}, \cite{G-K-Z-1} they 
assumed that $A$ is homogeneous (see Remark 
\ref{homo}). By removing this homogeneous 
condition, Adolphson \cite{A} 
generalized their $A$-hypergeometric systems 
to the confluent (i.e. irregular) case 
and proved many important results. 
However his construction of the confluent 
$A$-hypergeometric $\D$-modules is not 
given by the standard operations of 
$\D$-modules as in \cite{G-K-Z-1}, 
\cite{G-K-Z-2}. 
This leads us to some difficulties in  
obtaining the integral representations 
of confluent $A$-hypergeometric functions. 
Indeed, in the confluent case almost 
nothing is known about their global 
properties. In this paper, we 
construct Adolphson's 
confluent $A$-hypergeometric $\D$-modules 
as in \cite{G-K-Z-2}. 
Note that recently the same problem 
was solved more completely in Saito \cite{S1} and 
Schulze-Walther \cite{SW1}, \cite{SW2} 
by using commutative algebras. 
Our approach is based on 
sheaf-theoretical methods and totally 
different from theirs. 
Moreover we also construct 
an integral representation 
of the confluent 
$A$-hypergeometric functions 
by using the theory of 
rapid decay homology groups 
introduced recently 
in Hien \cite{H-1} and \cite{H-2}. 
Recall that $A=\{ a(1), a(2), \ldots , 
a(N)\} \subset \ZZ^{n}$ is a finite subset 
of a lattice $\ZZ^{n}$ and Adolphson's 
confluent $A$-hypergeometric system 
is defined on $\CC^A=\CC^N_z$. 
Then our integral representation of 
confluent $A$-hypergeometric functions 
\begin{equation}\label{IRF} 
u(z)= \int_{\gamma^z} 
\exp (\sum_{j=1}^N z_j x^{a(j)}) 
x_1^{c_1-1} \cdots x_n^{c_n-1} 
dx_1 \wedge \cdots \wedge dx_n  
\end{equation}
coincides with the one in Adolphson 
\cite[Equation (2.6)]{A}, 
where $\gamma =\{ \gamma^z \}$ is a family of 
real $n$-dimensional 
topological cycles $\gamma^z$ 
in the algebraic torus $T=(\CC^*)^n_x$ 
on which the function 
$\exp (\sum_{j=1}^N z_j x^{a(j)}) 
x_1^{c_1-1} \cdots x_n^{c_n-1}$ 
decays rapidly at infinity. 
More precisely $\gamma^z$ is an element of 
Hien's rapid decay homology group. 
See Sections \ref{sec:3} and \ref{sec:4} 
for the details. Adolphson used the 
formula \eqref{IRF} without giving any 
geometric condition on the cycles 
$\gamma^z$ nor proving the convergence 
of the integrals. In our Theorem \ref{Main} 
we could give a rigorous justification to 
Adolphson's formula \cite[Equation (2.6)]{A} 
by using rapid decay homology cycles. This integral 
representation can be 
considered as a natural generalization 
of those for the classical Bessel and 
Airy functions etc. Note that in the case of 
hypergeometric functions associated to 
hyperplane arrangements the same problem 
was precisely studied by Kimura-Haraoka-Takano 
\cite{K-H-T}. 
We hope that our geometric construction 
would be useful in the explicit study of 
Adolphson's confluent $A$-hypergeometric functions. 
In the proof of Theorem \ref{Main}, we shall use 
the method of toric compactifications 
introduced in \cite{L-S} 
and \cite{M-T-3} for the study of 
geometric monodromies of polynomial maps. 
Moreover we introduce 
Proposition \ref{RTH} which enables us to calculate 
Hien's rapid decay homologies 
by usual relative twisted 
homologies. By Proposition \ref{RTH} and 
Lemmas \ref{EC} and \ref{ECL} we can 
calculate the rapid decay homologies very 
explicitly in many cases. 
Let $\Delta \subset \RR^n$ be the convex 
hull of $A \cup \{ 0 \}$ in $\RR^n$ and 
$h_z: T=( \CC^*)^n_x \longrightarrow \CC$ 
the Laurent polynomial on $T$ defined by 
$h_z(x)= \sum_{j=1}^N z_j x^{a(j)}$. 
Then in Section \ref{sec:5}, assuming the 
condition $0 \in \Int ( \Delta )$ and 
using the twisted Morse theory we construct also a 
natural basis of the rapid 
decay homology group indexed by the 
critical points of $h_z$. 
Furthermore we apply it to obtain a 
precise formula for the asymptotic expansions 
at infinity of Adolphson's 
confluent $A$-hypergeometric functions. 
The formula that we obtain in Theorem \ref{ASE} 
will be very similar to that of the 
classical Bessel functions. Finally in 
Sections \ref{sec:6} and \ref{sec:7}, 
removing the condition $0 \in \Int ( \Delta )$  
we construct another natural basis of the rapid 
decay homology group. We thus partially solve 
the famous problem in Gelfand-Kapranov-Zelevinsky 
\cite{G-K-Z-2} of constructing a basis of 
the twisted homology group in their integral 
representation of non-confluent $A$-hypergeometric 
functions, in the more general case of 
confluent ones. Moreover, recently in \cite{A-E-T} 
(a slight modification of) 
this result was effectively used 
for the study of the 
monodromies at infinity of confluent 
$A$-hypergeometric functions. See \cite{A-E-T} 
for the details. 

\bigskip
\noindent{\bf Acknowledgement:} We thank Professors 
K. Ando, Y. Haraoka, T. Kohno, A. Pajitnov, 
N. Takayama and S. Tanabe 
for useful discussions during the preparation of 
this paper.

\section{Adolphson's results}\label{sec:2}

First of all, we recall the definition 
of the confluent $A$-hypergeometric systems 
introduced by Adolphson \cite{A} 
and their important properties.  In this paper, 
we essentially follow the terminology of 
\cite{H-T-T}. 
Let $A=\{ a(1), a(2), \ldots , 
a(N)\} \subset \ZZ^{n}$ be a finite subset 
of the lattice $\ZZ^{n}$. 
Assume that 
$A$ generates $\ZZ^{n}$ as in 
\cite{G-K-Z-1} and \cite{G-K-Z-2}. 
Following \cite{A} we denote by $\Delta$ the convex 
hull ${\rm conv} (A \cup \{ 0 \} )$ 
of $A \cup \{ 0\}$ in $\RR^{n}$. 
By definition $\Delta$ is an $n$-dimensional 
polytope. Let $c=(c_1, \ldots, c_n) \in 
\CC^{n}$ be a parameter vector. 
Moreover consider the 
$n \times N$ integer matrix
\begin{equation}
A:=\begin{pmatrix}
 {^t a(1)} & {^t a(2)} & \cdots & 
{^t a(N)} \end{pmatrix}=
(a_{i,j})_{1 \leq i \leq n, 1 \leq j \leq N}
 \in M(n, N, \ZZ)
\end{equation}
whose $j$-th column is ${^t a(j)}$.
Then Adolphson's confluent $A$-hypergeometric 
system on $X= \CC^A=\CC_z^{N}$ 
associated with the parameter vector 
$c=(c_1, \ldots, c_n) \in \CC^{n}$ is 
\begin{gather}
\left(\sum_{j=1}^{N} a_{i,j} 
z_j\frac{\partial}{\partial 
z_j}+c_i \right)
u(z)=0 \hspace{5mm} (1 \leq i \leq n), \\
\left\{ \prod_{\mu_j>0} \left(\frac{\partial}{\partial 
z_j}\right)^{\mu_j} -\prod_{\mu_j<0} 
\left(\frac{\partial}{\partial 
z_j}\right)^{-\mu_j} \right\} u(z)=0 
\hspace{5mm} (\mu \in {\rm Ker} A \cap 
\ZZ^{N}). 
\end{gather} 

\begin{remark}\label{homo} 
The above $A$-hypergeometric system was 
introduced first by Gelfand-Graev-Kapranov-Zelevinsky 
\cite{G-G-Z}, \cite{G-K-Z-1} under the homogeneous 
condition on $A$ i.e. when there exists 
a linear functional $l: \RR^n \longrightarrow \RR$ 
such that $l( \ZZ^n)= \ZZ$ and $A \subset l^{-1}(1)$. 
In this case, Hotta \cite{Hotta} proved that 
it is regular holonomic i.e. non-confluent. 
\end{remark}

\begin{remark}
In \cite{A} Adolphson does not assume that 
$A$ generates $\ZZ^n$. However we need this 
condition to obtain a geometric construction 
of his confluent $A$-hypergeometric systems. 
Even when $A$ does not generate $\ZZ^n$, 
by a suitable linear coordinate change of 
$\RR^n$ we can get an equivalent system for 
$A^{\prime} \subset \ZZ^n$ 
and $c^{\prime} \in \CC^n$ such 
that $A^{\prime}$ generates $\ZZ^n$. Namely 
our condition is not restrictive at all. 
\end{remark}
Let $D(X)$ be the Weyl algebra over $X$ 
and consider the differential operators 
\begin{eqnarray}
Z_{i,c} & := & \sum_{j=1}^{N} a_{ij} z_j
\frac{\partial}{\partial z_j}+c_i 
\hspace{5mm}(1\leq i\leq n),\\
\square_{\mu} 
& := & \prod_{\mu_j>0}\left(\frac{\partial}{\partial 
z_j}\right)^{\mu_j} -\prod_{\mu_j<0} \left(
\frac{\partial}{\partial 
z_j}\right)^{-\mu_j}\hspace{5mm} ( \mu 
\in {\rm Ker} A \cap 
\ZZ^{N})
\end{eqnarray}
in it. Then the above system is 
naturally identified with 
the left $D(X)$-module 
\begin{equation}
M_{A, c}=
D(X) / \left(\sum_{1 \leq i \leq n} 
D(X) Z_{i,c} + \sum_{\mu \in {\rm Ker} A 
\cap \ZZ^{N}} D(X) \square_{\mu} \right). 
\end{equation} 
Let $\D_{X}$ be the sheaf 
of differential operators 
over the algebraic variety $X$ and 
define a coherent $\D_{X}$-module by 
\begin{equation}
\M_{A, c}=
\D_{X} / \left(\sum_{1 \leq i \leq n} 
\D_{X} Z_{i,c} + \sum_{\mu \in {\rm Ker} A 
\cap \ZZ^{N} } \D_{X} \square_{\mu}\right). 
\end{equation}
Then we have an isomorphism 
$M_{A, c} \simeq \Gamma (X; \M_{A, c})$ 
(see \cite[Proposition 1.4.4]{H-T-T}). 
Adolphson \cite{A} proved that $\M_{A, c}$ 
is holonomic. In fact, he proved the 
following more precise result. 

\begin{definition}\label{AND} 
(Adolphson \cite[page 274]{A}, see also 
\cite{Oka}) For $z \in X=\CC^A$ we 
say that the Laurent polynomial 
$h_z(x)= \sum_{j=1}^N z_jx^{a(j)}$ is 
non-degenerate if for any face $\Gamma$ of 
$\Delta$ not containing the origin 
$0 \in \RR^n$ we have 
\begin{equation}\label{E-111} 
\left\{ x \in T=(\CC^*)^n \ | \ 
\frac{\partial h_z^{\Gamma}}{\partial x_1}(x)
= \cdots \cdots = 
\frac{\partial h_z^{\Gamma}}{\partial x_n}
(x)=0 \right\} = \emptyset, 
\end{equation}
where we set $h_z^{\Gamma}(x)=
\sum_{j:a(j) \in \Gamma} z_jx^{a(j)}$. 
\end{definition}

\begin{remark}
Since in the definition above 
we consider only faces $\Gamma \prec \Delta$ 
such that $\dim \Gamma \leq n-1$ and $h_z$ are 
quasi-homogeneous, 
our condition \eqref{E-111} is equivalent to 
the weaker one in \cite[page 274]{A}. 
See Kouchnirenko \cite[Definition 1.19]{Kushnirenko}. 
\end{remark}
Let $\Omega \subset X$ 
be the Zariski open subset 
of $X$ consisting of $z \in X= \CC^A$ 
such that the Laurent polynomial 
$h_z(x)= \sum_{j=1}^N z_jx^{a(j)}$ 
is non-degenerate. 
Then Adolphson's result 
in \cite[Lemma 3.3]{A} asserts 
that the holonomic $\D_X$-module 
$\M_{A, c}$ is an integrable 
connection on $\Omega$ (i.e. the 
characteristic variety of $\M_{A, c}$ 
is contained in the zero section 
of the cotangent bundle $T^*\Omega$). 
Now let $X^{\an}$ (resp. $\Omega^{\an}$) 
be the underlying 
complex analytic manifold 
of $X$ (resp. $\Omega$) and consider 
the holomorphic solution complex 
${\rm Sol}_X(\M_{A, c}) \in \BDC (X^{\an})$ 
of $\M_{A, c}$ defined by 
\begin{equation}
{\rm Sol}_X (\M_{A, c})=\rhom_{\D_{X^{\an}}}
((\M_{A, c})^{\an}, \sho_{X^{\an}})  
\end{equation}
(see \cite{H-T-T} for the details). 
Then by the above Adolphson's result, 
${\rm Sol}_X(\M_{A, c})$ is a local 
system on $\Omega^{\an}$. Moreover 
he proved the following remarkable result. 
Let $\Vol_{\ZZ}(\Delta ) \in \ZZ$ be the 
normalized (or simplicial) $n$-dimensional 
volume of $\Delta$ i.e. $n!$ times the Euclidean 
volume of $\Delta \subset \RR^n$ with respect to 
the canonical embeddings 
$\ZZ^n \subseteq \QQ^n \subseteq \RR^n$. 

\begin{theorem}\label{ADL} 
(Adolphson \cite[Corollary 5.20]{A}) 
Assume that the parameter vector 
$c \in \CC^n$ is semi-nonresonant 
(for the definition see \cite[page 284]{A}). Then 
the rank of the local system 
$H^0 {\rm Sol}_X (\M_{A, c})|_{\Omega^{\an}}$ on 
$\Omega^{\an}$ is equal to $\Vol_{\ZZ}(\Delta ) \in \ZZ$. 
\end{theorem}
This is a generalization of 
the famous result of Gelfand-Kapranov-Zelevinsky 
in \cite{G-K-Z-1} to 
the confluent case. Later Matusevich-Miller-Walther 
\cite{M-M-W} generalized it further by showing 
that the holonomic rank of $\M_{A, c}$ 
does not jump outside the union of 
finitely many subspaces of $\CC^n$ of 
codimension at least two. We call the sections 
of the local system 
$H^0 {\rm Sol}_X (\M_{A, c}
)|_{\Omega^{\an}}$ 
confluent $A$-hypergeometric functions 
(associated to the parameter 
$c \in \CC^n$).

\section{Hien's rapid decay 
homologies}\label{sec:3}

In this section, we review Hien's 
theory of rapid decay 
homologies invented  
in \cite{H-1} and \cite{H-2}. 
For the theory of twisted homology groups 
we refer to Aomoto-Kita \cite{A-K} and 
Pajitnov \cite{Paj}. 
Let $U$ be a smooth quasi-projective variety 
of dimension $n$ and $(\E , \nabla )$ 
($\nabla : \E \longrightarrow 
\Omega_U^1 \otimes_{\sho_U} \E$) 
an integrable connection on it. 
We consider $(\E , \nabla )$ as a left 
$\D_U$-module and set 
\begin{equation}
{\rm DR}_U (\E )=\Omega_{U^{\an}}
\otimes_{\D_{U^{\an}}}^L \E^{\an} 
\simeq \Omega_{U^{\an}}^{\bullet} 
\otimes_{\sho_{U^{\an}}} \E^{\an}[n]. 
\end{equation}
Assume that 
$i: U \hookrightarrow Z$ is a smooth projective 
compactification of $U$ such that 
$D=Z \setminus U$ is a normal crossing divisor 
and the extension $i_*\E$ of $\E$ to $Z$ 
admits a good lattice in the sense of 
Mochizuki \cite{Mochizuki}. 
Such a good compactification of $U$ for $(\E , \nabla )$ 
always exists by the fundamental 
theorem recently established by 
Mochizuki \cite{Mochizuki}. Now let $\pi : \tilde{Z} 
\longrightarrow Z^{\an}$ be the real oriented 
blow-up of $Z^{\an}$ in \cite{H-1}, \cite{H-2} 
and set $\tilde{D} = 
\pi^{-1}(D^{\an})$. Recall that $\pi$ induces 
an isomorphism $\tilde{Z} \setminus 
\tilde{D} \simto Z^{\an} \setminus D^{\an}$. 
More precisely, for each point 
$q \in D^{\an}$ by taking a local coordinate 
$(x_1, \ldots, x_n)$ on a neighborhood of 
$q$ such that $q=(0, \ldots, 0)$ and 
$D^{\an}=\{ x_1 \cdots x_k=0 \}$ the 
morphism $\pi$ is explicitly given by 
\begin{eqnarray}
([0, \varepsilon ) \times S^1)^k 
\times B(0; \varepsilon )^{n-k} 
& \longrightarrow & 
B(0; \varepsilon )^{k} 
\times B(0; \varepsilon )^{n-k}
\\
( \{ (r_i, e^{\sqrt{-1} \theta_i}) \}_{i=1}^k, 
x_{k+1}, \ldots, x_n) 
 & \longmapsto & 
( \{ r_i e^{\sqrt{-1} \theta_i} \}_{i=1}^k, 
x_{k+1}, \ldots, x_n), 
\end{eqnarray}
where we set $B(0; \varepsilon )=
\{ x \in \CC \ | \ |x|< \varepsilon \}$ 
for $\varepsilon >0$. For $p \geq 0$ 
and a subset $B \subset \tilde{Z}$ 
denote by $S_p(B)$ the $\CC$-vector 
space generated by the piecewise smooth 
maps $c: \Delta^p \longrightarrow 
B$ from the $p$-dimensional simplex 
$\Delta^p$. We denote by 
$\CAC^{-p}_{\tilde{Z}, \tilde{D}}$ 
the sheaf on $\tilde{Z}$ associated to 
the presheaf 
\begin{equation}
V \longmapsto S_p(\tilde{Z}, 
(\tilde{Z} \setminus V) \cup \tilde{D}) 
=S_p(\tilde{Z}) /  S_p 
((\tilde{Z} \setminus V) \cup \tilde{D}). 
\end{equation}
Namely $\CAC^{-p}_{\tilde{Z}, \tilde{D}}$ 
is the sheaf of the relative $p$-chains 
on the pair $(\tilde{Z}, \tilde{D})$. 
Now let $\LL :=H^{-n} {\rm DR}_U( \E )
= \Ker \{ \nabla^{\an}: \E^{\an} 
\longrightarrow \Omega^1_{U^{\an}} 
\otimes_{\sho_{U^{\an}}} \E^{\an} \}$ 
be the sheaf of horizontal sections 
of the analytic connection 
$(\E^{\an}, \nabla^{\an})$ and 
$\iota : U^{\an} \hookrightarrow 
\tilde{Z}$ the inclusion. Then 
$\iota_* \LL$ is a local system on 
$\tilde{Z}$. We define the sheaf 
$\CAC^{-p}_{\tilde{Z}, \tilde{D}}( 
\iota_* \LL )$ 
of the relative twisted $p$-chains 
on the pair $(\tilde{Z}, \tilde{D})$ 
with coefficients in $\iota_* \LL$ by 
$\CAC^{-p}_{\tilde{Z}, 
\tilde{D}}( \iota_* \LL )
=\CAC^{-p}_{\tilde{Z}, \tilde{D}} 
\otimes_{\CC_{\tilde{Z}}} \iota_* \LL$. 

\begin{definition} 
(Hien \cite{H-1} and \cite{H-2}) 
A section $\gamma =c \otimes s \in \Gamma 
(V; \CAC^{-p}_{\tilde{Z}, \tilde{D}}( 
\iota_* \LL ))$ is called a rapid 
decay chain if for any point 
$q \in c(\Delta^p) \cap \tilde{D} 
\cap V$ the following condition holds: 
\par \indent 
In a local coordinate 
$(x_1, \ldots, x_n)$ on a neighborhood of 
$q$ in $Z$ such that $q=(0, \ldots, 0)$ and 
$D^{\an}=\{ x_1 \cdots x_k=0 \}$ 
by taking a local trivialization 
$(i_*\E )^{\an} \simeq \oplus_{i=1}^r 
\sho_{Z^{\an}}(*D^{\an})e_i$ 
with respect to a basis $e_1, \ldots, e_r$ 
and setting $s=\sum_{i=1}^r f_i \cdot 
\iota_* i^{-1} e_i$ 
($f_i \in \iota_* \sho_{Z^{\an}}$), 
for any $1 \leq i \leq r$ and 
$N=(N_1, \ldots, N_k) \in \NN^k$ 
there exists $C_N>0$ such that 
\begin{equation}
 | f_i(x) | \leq 
C_N |x_1|^{N_1} \cdots |x_k|^{N_k}
\end{equation}
for any $x \in (c(\Delta^p) \setminus 
\tilde{D}) \cap V$ with small $|x_1|, 
\ldots, |x_k|$. 
\par \indent In particular, 
if $c(\Delta^p) \cap \tilde{D} 
\cap V = \emptyset$ we do not impose any 
condition on $s \in \iota_* \LL$. 
\end{definition} 
Note that this definition does not 
depend on the local coordinate 
$(x_1, \ldots, x_n)$ 
nor the local trivialization 
$(i_*\E )^{\an} \simeq \oplus_{i=1}^r 
\sho_{Z^{\an}}(*D^{\an})e_i$. 
We denote by $\CAC^{\rd ,-p}_{\tilde{Z}, 
\tilde{D}}( \iota_* \LL )$ the subsheaf 
of $\CAC^{-p}_{\tilde{Z}, \tilde{D}}( 
\iota_* \LL )$ consisting of rapid decay 
chains. According to Hien 
\cite{H-1} and \cite{H-2}, 
$\CAC^{\rd ,-p}_{\tilde{Z}, 
\tilde{D}}( \iota_* \LL )$ is a fine 
sheaf. Then we obtain a complex of 
fine sheaves on $\tilde{Z}$: 
\begin{equation}
\CAC^{\rd , \bullet}_{\tilde{Z}, 
\tilde{D}}( \iota_* \LL ) = 
\left[ \cdots \longrightarrow 
\CAC^{\rd ,-(p+1)}_{\tilde{Z}, 
\tilde{D}}( \iota_* \LL )  \longrightarrow 
\CAC^{\rd ,-p}_{\tilde{Z}, 
\tilde{D}}( \iota_* \LL )  \longrightarrow 
\CAC^{\rd ,-(p-1)}_{\tilde{Z}, 
\tilde{D}}( \iota_* \LL ) 
\longrightarrow \cdots \right].
\end{equation}

\begin{definition} 
(Hien \cite{H-1} and \cite{H-2}) 
For $p \in \ZZ$ we set 
\begin{equation}
H_p^{\rd}(U; \E ):=H^{-p} \Gamma 
(\tilde{Z} ; \CAC^{\rd , \bullet}_{\tilde{Z}, 
\tilde{D}}( \iota_* \LL ) )
\end{equation}
and call it the $p$-th rapid decay 
homology group associated to 
the integrable connection $\E$. 
\end{definition} 
In \cite{H-2} Hien proved that $H_p^{\rd}(U; \E )$ 
is isomorphic to the dual of the $p$-th 
algebraic de Rham cohomology of the dual 
connection $\E^*$ of $\E$. 
In this paper, we use only some 
special integrable connections 
$( \E, \nabla )$ as the following example. 

\begin{example} 
Let $U \simeq \CC_x^*$ and $\E = \sho_U 
\exp (-h(x))x^{-c}$, where $h(x)=\sum_{i \in \ZZ} 
a_ix^i$ ($a_i \in \CC$) is a Laurent polynomial 
and $c \in \CC$. As usual we endow $\E = \sho_U 
\exp (-h(x))x^{-c}$ with the connection 
$\nabla : \E \longrightarrow \Omega_U^1 
\otimes_{\sho_U} \E$ defined by 
\begin{equation}
\nabla \{ f \exp (-h(x))x^{-c} \} 
=df \otimes \exp (-h(x))x^{-c} - 
(dh+ \frac{c}{x} dx ) \otimes f \exp (-h(x))x^{-c}
\end{equation}
for $f \in \sho_U$. Then we have 
$\LL = H^{-1} {\rm DR}_U( \E ) \simeq 
\CC_{U^{\an}} \exp (h(x))x^c 
\subset \sho_{U^{\an}}$. In this case, 
to define the rapid decay homology groups 
$H_p^{\rd}(U; \E )$ we consider 
(relative) twisted chains on which 
the function $\exp (h(x))x^c$ decays 
rapidly at infinity. 
\end{example} 
If $\E \simeq \sho_U ( \frac{1}{g})$ and 
we have an isomorphism $\LL =H^{-n} {\rm DR}_U( \E )
\simeq \CC_{U^{\an}}g$ ($\subset \sho_{U^{\an}}$) 
for a possibly multi-valued 
holomorphic function $g: U^{\an} 
\longrightarrow \CC$ as 
the example above, we call 
$H_p^{\rd}(U; \E )$ the $p$-th rapid decay 
homology group associated to 
the function $g$. 
In the special case where $g(x)=\exp (h(x)) g_0(x)$ 
for a meromorphic function $h$ on $Z^{\an}$ 
with poles in $D^{\an}$ and 
a (possibly multi-valued) function $g_0$ on $U^{\an}$ 
such that at each point of $Z^{\an}$ 
there exists a local coordinate $x=(x_1, 
\ldots, x_n)$ satisfying $g_0(x)= 
x_1^{c_1} \cdots x_n^{c_n}$ ($c_i \in \CC$) 
and $D= \{ x_1 \cdots x_k =0 \}$, we 
shall give a purely topological 
interpretation of $H_p^{\rd}(U; \E )$. 
Since $Z$ is a good compactification of $U$ 
for $\E$, the meromorphic function 
$h$ has no point of indeterminacy on 
the whole $Z^{\an}$ (see 
\cite[Section 2.1]{H-R}).  
By $\iota : U^{\an} \hookrightarrow \tilde{Z}$ 
we consider $U^{\an}$ as an open 
subset of $\tilde{Z}$ and set 
\begin{equation}
P= \tilde{D} \cap 
\overline{ \{ x \in U^{\an}  \ | \ 
{\rm Re} h(x) \geq 0 \} }. 
\end{equation}
Let $D=D_1 \cup \cdots \cup D_d$ be the 
irreducible decomposition of $D$. For 
$1 \leq i \leq d$ let $b_i \in \ZZ$ 
be the order of the meromorphic function 
$h$ along $D_i$. 
If $b_i \geq 0$ we say that the irreducible 
component $D_i$ is irrelevant. 
Namely along a relevant component $D_i$ 
the function $h$ has a pole of 
order $-b_i >0$. Denote 
by $D^{\prime}$ the union of the 
irrelevant components of $D$. 
Then we set $Q=\tilde{D} \setminus 
\{ P \cup \pi^{-1}(D^{\prime})^{\an} \}$. Note 
that $Q$ is an open subset of $\tilde{D}$ 
(i.e. the set of the rapid decay directions 
of the function $g$ in $\tilde{D}$). 

\begin{proposition}\label{RTH} 
In the situation as above (i.e. 
$\E^* = \sho_U \exp (h(x)) g_0(x)$), 
we have an isomorphism 
\begin{equation}
H_p^{\rd}(U ; \E ) \simeq 
H_p(U^{\an} 
 \cup Q, Q; \iota_*(\CC_{U^{\an}}g_0 ) )
\end{equation}
for any $p \in \ZZ$, 
where the right hand side is the $p$-th 
relative twisted homology group of the pair 
$(U^{\an} \cup Q, Q)$ with coefficients 
in the rank-one local 
system $\iota_*(\CC_{U^{\an}}g_0)$ 
on $\tilde{Z}$ (see \cite{A-K}). 
\end{proposition} 
\begin{proof}
Since the function $\exp (h(x))$ is single-valued, 
we have an isomorphism $\LL \simeq 
\CC_{U^{\an}}g_0$. First let us consider the 
case $n=1$. Locally we may assume that 
$U= \CC^*$, $Z= \CC_x= \CC^* \sqcup \{ 0 \}$, 
$D= \{ x=0 \} = \{ 0 \} \subset Z$ and 
$h(x)=x^{-m}$ $(m>0)$. Let $\pi : \tilde{Z} 
\longrightarrow Z^{\an}$ be the real oriented 
blow-up of $Z^{\an}$ 
along $D^{\an}$. In this case we have 
$\tilde{D}= \pi^{-1}(D^{\an}) \simeq S^1$ 
and $\tilde{Z} = \{ (r, e^{\sqrt{-1} \theta}) \ | \ 
r \geq 0 \} \simeq 
[0, \infty ) \times S^1$. For $1 \leq i \leq m$ 
and sufficiently small $\e >0$ we set 
\begin{equation}
Q_i^{\e}=\{ e^{\sqrt{-1} \theta} \in \tilde{D} 
\simeq S^1 \ | \ 
\frac{(2i-\frac{3}{2}) \pi}{m}- \e 
< \theta < 
\frac{(2i-\frac{1}{2}) \pi}{m}+ \e \}, 
\end{equation}
and $Q^{\e}= \cup_{i=1}^m Q_i^{\e} 
\subset \tilde{D}$. 
Note that $Q^{\e} \subset \tilde{D}$ contains all the rapid 
decay directions of $g(x)=\exp (h(x)) g_0(x)$ 
in $\tilde{D}$. Now let us consider the two 
topological subspaces $U^{\an} \cup Q^{\e}$ 
and $\tilde{D}$ of $\tilde{Z}$. We 
patch them on their 
intersection $Q^{\e}$ and construct 
a new topological space 
as follows. By identifying the points 
of $Q^{\e} \subset U^{\an} \cup Q^{\e}$ and 
those of $Q^{\e} \subset \tilde{D}$ naturally, 
we obtain a quotient space 
$\tl{Z^{\e}}$ of the disjoint union 
$(U^{\an} \cup Q^{\e}) \sqcup \tilde{D}$. 
Recall that $\tl{Z^{\e}}$ is endowed with 
the strongest topology for which the quotient 
map $(U^{\an} \cup Q^{\e}) \sqcup \tilde{D} 
\longrightarrow \tl{Z^{\e}}$ is continuous. 
Note also that $\tilde{D}$ is naturally 
identified with a close subspace of 
$\tl{Z^{\e}}$. We denote the local 
system on $\tl{Z^{\e}}$ naturally 
constructed from $\iota_* \LL$ by 
the same letter $\iota_* \LL$. 
For $p \in \ZZ$ let 
$S_p( \tl{Z^{\e}}, \tilde{D} ; \iota_* \LL )$ 
be the $\CC$-vector space of the 
twisted (piecewise smooth) relative $p$-chains 
on the pair $(\tl{Z^{\e}}, \tilde{D})$ 
with coefficients in  
$\iota_* \LL$. Then by the definition of 
rapid decay chains, for any $p \in \ZZ$ 
we obtain a natural morphism 
\begin{equation}
\Gamma (\tilde{Z} ; \CAC^{\rd ,-p}_{\tilde{Z}, 
\tilde{D}}( \iota_* \LL ) ) \longrightarrow 
S_p( \tl{Z^{\e}}, \tilde{D} ; \iota_* \LL ). 
\end{equation}
We can easily show that the chain map 
\begin{equation}
\Gamma (\tilde{Z} ; \CAC^{\rd , \bullet}_{\tilde{Z}, 
\tilde{D}}( \iota_* \LL ) ) \longrightarrow 
S_{- \bullet}( \tl{Z^{\e}}, \tilde{D} ; \iota_* \LL ) 
\end{equation}
obtained in this way 
is a homotopy equivalence. 
Indeed, its homotopy inverse $S_{- \bullet}( 
\tl{Z^{\e}}, \tilde{D} ; \iota_* \LL ) 
\longrightarrow \Gamma (\tilde{Z} ; 
\CAC^{\rd , \bullet}_{\tilde{Z}, 
\tilde{D}}( \iota_* \LL ) )$ can be constructed 
by smooth deformations of chains in 
$S_{- \bullet}( 
\tl{Z^{\e}}, \tilde{D} ; \iota_* \LL )$ 
in the angular direction $\theta = 
{\rm arg} x$. We can construct them by a smooth 
vector field on $\tilde{Z}$. 
Hence we obtain an isomorphism 
\begin{equation}\label{EQ-1} 
H^{-p} \Gamma 
(\tilde{Z} ; \CAC^{\rd , \bullet}_{\tilde{Z}, 
\tilde{D}}( \iota_* \LL ) ) \simto 
H_p( \tl{Z^{\e}} , \tilde{D} ; \iota_* \LL )
\end{equation}
for any $p \in \ZZ$. 
Moreover by excision and 
homotopy, we have an isomorphism 
\begin{equation}\label{EQ-2} 
H_p( \tl{Z^{\e}} , \tilde{D} ; \iota_* \LL ) 
\simto 
H_p( U^{\an} \cup Q^{\e}, Q^{\e} ; \iota_* \LL ) 
\end{equation}
for any $p \in \ZZ$. Combining \eqref{EQ-1} 
with \eqref{EQ-2}, we obtain an isomorphism 
\begin{equation} 
H^{-p} \Gamma 
(\tilde{Z} ; \CAC^{\rd , \bullet}_{\tilde{Z}, 
\tilde{D}}( \iota_* \LL ) )
\simto  
H_p( U^{\an} \cup Q^{\e}, Q^{\e} ; \iota_* \LL ) 
\end{equation}
for any $p \in \ZZ$. 

\medskip \par 
Finally let 
us consider the case $n \geq 2$. First 
we assume that for some $1 \leq k \leq n$ 
we have $Z= \CC^n_x$, 
$D= \{ x_1 \cdots x_k =0 \} \subset Z$, 
$U=Z \setminus D$ and 
$h(x)=x_1^{-m_1} \cdots x_k^{-m_k}$ 
$(m_i>0)$.  Let $\pi : \tilde{Z} 
\longrightarrow Z^{\an}$ be the real oriented 
blow-up of $Z^{\an}$ along $D^{\an}$. 
In this case we have $\tilde{Z}=
\{ ( \{ (r_i, e^{\sqrt{-1} \theta_i}) \}_{i=1}^k, 
x_{k+1}, \ldots, x_n) \ | \ r_i \geq 0 \} 
\simeq 
([0, \infty ) \times S^1)^k \times \CC^{n-k}$. 
For sufficiently small $\e >0$ we define 
an open subset $Q^{\e} \subset \tilde{D}$ by 
$( \{ (r_i, e^{\sqrt{-1} \theta_i}) \}_{i=1}^k, 
x_{k+1}, \ldots, x_n)  \in  Q^{\e}$ 
$\Longleftrightarrow$ 
${\rm Re} \ 
e^{\sqrt{-1} (m_1 \theta_1+ \cdots +m_k \theta_k)} 
  <  \e | {\rm Im} \ 
e^{\sqrt{-1} (m_1 \theta_1+ \cdots +m_k \theta_k)} |$ 
for $( \{ (r_i, e^{\sqrt{-1} \theta_i}) \}_{i=1}^k, 
x_{k+1}, \ldots, x_n) \in \tilde{D}$. 
Then $Q^{\e}$ contains all the rapid 
decay directions of 
$g(x)=\exp (h(x)) g_0(x)$ 
in $\tilde{D}$. As in the case $n=1$, 
by smooth deformations of chains and excision 
etc. we obtain an isomorphism 
\begin{equation} 
H^{-p} \Gamma 
(\tilde{Z} ; \CAC^{\rd , \bullet}_{\tilde{Z}, 
\tilde{D}}( \iota_* \LL ) ) \simto  
H_p( U^{\an} \cup Q^{\e}, Q^{\e} ; \iota_* \LL )
\end{equation}
for any $p \in \ZZ$. The general case can be 
proved similarly by patching 
local smooth deformations of chains 
(smooth vector fields) as above by 
a partition of unity. 
This completes the proof. 
\end{proof}

The following lemma will be used in 
Section \ref{sec:4}. 
\begin{lemma}\label{EC} 
In the situation of Proposition \ref{RTH}, 
for a point $q \in D^{\an}$ 
let $k \geq 0$ be the number of 
the relevant irreducible components of $D^{\an}$ 
passing through $q$. Assume that $k \geq 2$. 
Then for a small open neighborhood $V$ of $q$ 
in $Z^{\an}$ we have 
\begin{equation}\label{Eule} 
\sum_{p \in \ZZ} (-1)^p \dim 
H_p( (V \cap U^{\an}) \cup ( \pi^{-1}(V) \cap Q) , 
( \pi^{-1}(V) \cap Q) ; \iota_*(\CC_{U^{\an}}g_0 ) )
=0.
\end{equation}
\end{lemma}
\begin{proof}
The problem being local, 
we may assume that for some $k \leq l \leq n$ 
we have $V=Z= \CC^n_x$, 
$D= \{ x_1 \cdots x_l =0 \} \subset Z$, 
$U=Z \setminus D$ and $q=0 \in D^{\an}$. 
For simplicity here we consider the case where $k=l$ 
and $h(x)=x_1^{-m_1} \cdots x_k^{-m_k}$ 
$(m_i>0)$.  Then we have the product decomposition 
$\tilde{Z}=
\{ ( \{ (r_i, e^{\sqrt{-1} \theta_i}) \}_{i=1}^k, 
x_{k+1}, \ldots, x_n) \ | \ r_i \geq 0 \} 
\simeq 
([0, \infty ) \times S^1)^k \times \CC^{n-k}$ 
of $\tilde{Z}$ and by the projection 
$p_1 : \tilde{Z} 
\longrightarrow S^1$ defined by 
$( \{ (r_i, e^{\sqrt{-1} \theta_i}) \}_{i=1}^k, 
x_{k+1}, \ldots, x_n) \longmapsto 
e^{\sqrt{-1} \theta_1}$ the two manifolds 
$U^{\an}, Q \subset \tilde{Z}$ are fiber 
bundles over $S^1$. Let $S^1= \cup_{i=1}^r I_i$ 
be an open covering of $S^1$ such that the 
restrictions $p_1^{-1}(I_i) \cap U^{\an} 
\longrightarrow I_i$ and $p_1^{-1}(I_i) \cap Q 
\longrightarrow I_i$ of the above two 
fiber bundles to $I_i
\subset S^1$ are isomorphic 
to the trivial ones over $I_i$ for 
any $1 \leq i \leq r$. Then by 
the Mayer-Vietoris exact sequences for the relative 
twisted homology groups 
\begin{equation}
H_p( (p_1^{-1}(I_i) \cap U^{\an}) 
\cup ( p_1^{-1}(I_i) \cap Q) , 
( p_1^{-1}(I_i) \cap Q) ; \iota_*(\CC_{U^{\an}}g_0 ) ), 
\end{equation}
in the calculation of the left hand side 
of \eqref{Eule} we can replace the 
two fiber bundles by the trivial ones. 
By the K\"{u}nneth formula, the assertion 
follows immediately from the fact that 
the Euler characteristic of the circle $S^1$ 
is zero.  
\end{proof} 
In the sequel, we consider the more special 
case where $U=\CC^*_x$ and 
$( \E, \nabla )$ is an 
integrable connection on $U$ such that 
$\E^* = \sho_U \exp (h(x))x^c$ and 
$\LL = H^{-1} {\rm DR}_U( \E ) \simeq 
\CC_{U^{\an}} \exp (h(x))x^c$ for a 
Laurent polynomial $h(x)=\sum_{i \in \ZZ} 
a_ix^i$ ($a_i \in \CC$) and $c \in \CC$. 
Then we can take the projective 
line $\PP$ to be the good compactification $Z$ of 
$U=\CC^*_x$ for $( \E, \nabla )$. 
In this case, we have 
$D=Z \setminus U=D_1 \cup D_2$, where 
we set $D_1=\{ 0 \}$ and $D_2=\{ \infty \}$. 
For the real oriented 
blow-up $\pi : \tilde{Z} 
\longrightarrow Z^{\an}$ of $Z^{\an}$ 
the subset $\tilde{D} = 
\pi^{-1}(D^{\an})$ of $\tilde{Z}$ 
is a union of two circles 
$\tl{D_i}:=\pi^{-1}(D_i^{\an}) 
\simeq S^1$ ($i=1,2$). Moreover 
the open subset $Q \subset \tilde{D}$ 
is a union of open intervals in 
$\tl{D_1} \cup \tl{D_2} \simeq 
S^1 \cup S^1$. Let $NP(h) \subset \RR$ 
be the Newton polytope of $h$ i.e. the 
convex hull of the set $\{ i \in \ZZ \ | \ 
a_i \not= 0 \}$ in $\RR$. Finally denote 
by $\Delta \subset \RR$ the convex hull 
of $NP(h) \cup \{ 0 \}$ in $\RR$. 
Then by Proposition \ref{RTH} and 
Mayer-Vietoris exact sequences for relative 
twisted homology groups 
we can easily prove the following result. 

\begin{lemma}\label{ECL} 
In the situation as above (i.e. $U= \CC^*_x$ and 
$\E^* = \sho_U \exp (h(x))x^c$), we have 
\par \noindent (i) The dimension of 
the rapid decay homology group 
$H_p^{\rd}(U ; \E )$ is $\Vol_{\ZZ} 
(\Delta )$ if $p=1$ and zero otherwise. 
\par \noindent (ii) Assume that $\Delta =[-m, 0]$ 
(resp. $\Delta =[0, m]$) for some $m>0$. 
Then $Q \subset \tilde{D}$ is a union of 
$m$ open intervals $Q_1, Q_2, \ldots, Q_m$ 
in $\tl{D_1} \simeq S^1$ (resp. in 
$\tl{D_2} \simeq S^1$) and the first 
rapid decay homology group 
$H_1^{\rd}(U ; \E )$ has a basis formed by 
the $m$ elements 
\begin{equation}
[\gamma_i] \in H_1^{\rd}(U ; \E ) 
\qquad (i=1,2, \ldots, m), 
\end{equation}
where $\gamma_i$ is a $1$-dimensional 
twisted chain with values in $\iota_* \LL$ 
as in Figure $1$ below starting from a point in $Q_i$ and 
going directly to that in $Q_{i+1}$ 
(here we set $Q_{m+1}=Q_1$). 
\par \noindent (iii) Assume that 
$\Delta =[-m_1, m_2]$ for some $m_1, m_2 >0$. 
Then $Q \subset \tilde{D}$ is a union of 
open intervals $Q_1, Q_2, \ldots, Q_{m_1}$ 
in $\tl{D_1} \simeq S^1$ and the ones 
$Q_1^{\prime}, Q_2^{\prime}, 
\ldots, Q_{m_2}^{\prime}$ 
in $\tl{D_2} \simeq S^1$. 
If moreover $c \notin \ZZ$, then the first 
rapid decay homology group 
$H_1^{\rd}(U ; \E )$ has a basis formed by 
the $m_1+m_2$ elements 
\begin{equation}
[\gamma_i] \in H_1^{\rd}(U ; \E ) 
\qquad (i=1,2, \ldots, m_1) 
\end{equation}
and 
\begin{equation}
[\gamma_i^{\prime}] \in H_1^{\rd}(U ; \E ) 
\qquad (i=1,2, \ldots, m_2), 
\end{equation}
where $\gamma_i$ (resp. $\gamma_i^{\prime}$) 
is a $1$-dimensional 
twisted chain with values in $\iota_* \LL$ 
starting from a point in $Q_i$ (resp. 
$Q_i^{\prime}$) and 
going directly to that in $Q_{i+1}$ 
(resp. $Q_{i+1}^{\prime}$). 
\end{lemma}

\begin{center}
\includegraphics[height=6cm]{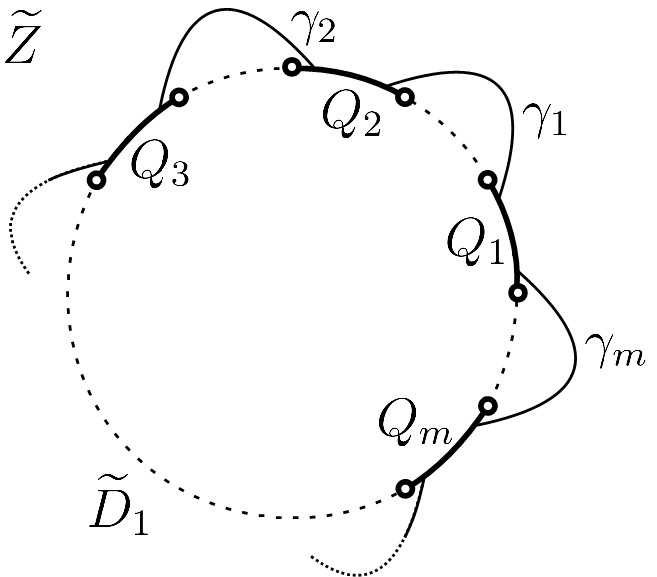}

Figure 1.
\end{center}


\section{A geometric construction of 
integral representations}\label{sec:4}

In this section we give a 
geometric construction of 
Adolphson's confluent $A$-hypergeometric 
$\D$-module $\M_{A, c}$ and apply it to 
obtain the integral representations of 
$A$-hypergeometric functions. 
Let $Y= \CC_{\zeta}^{N}$ be 
the dual vector space of 
$X= \CC^A=\CC_z^{N}$, where $\zeta$ is the dual 
coordinate of $z$. As in \cite{G-K-Z-2}, 
to $A \subset \ZZ^n$ we associate a 
morphism 
\begin{equation}
j: T=(\CC^*)^n_x \longrightarrow  
Y=\CC_{\zeta}^{N}
\end{equation}
defined by $x \longmapsto 
(x^{a(1)}, x^{a(2)}, \ldots, x^{a(N)})$. 
Since we assume here that 
$A$ generates $\ZZ^n$, $j$ is an embedding. 
Let $I \subset \CC [\zeta_1, \ldots, \zeta_N]$ 
be the defining ideal of the closure 
$\overline{j(T)}$ of 
$j(T) \subset Y$ in $Y$. Moreover denote by 
$D(Y)$ the Weyl algebra over $Y$. 
Then we have a ring isomorphism 
\begin{equation}
 \GFT : D(X) \simto D(Y) 
\end{equation}
defined by 
\begin{equation}
(\frac{\partial}{\partial z_j})^{ \GFT} 
=\zeta_j, 
\qquad 
(z_j)^{ \GFT}
= - \frac{\partial}{\partial \zeta_j} 
\qquad (j=1,2, \ldots, N). 
\end{equation}
We call $\GFT$ the Fourier transform 
(see Malgrange 
\cite{Malgrange} for the details). 
Via this $\GFT$, the Adolphson's system 
$M_{A,c}$ is 
transformed to the one 

\begin{equation}\label{DSY} 
(Z_{i,c})^{ \GFT} v(\zeta )  =  0 
\quad (1 \leq i \leq n), \qquad 
f(\zeta ) v(\zeta )  =  0 
\quad (f \in I) 
\end{equation}
on $Y= \CC_{\zeta}^{N}$. Note that this 
system has no non-zero 
holomorphic solution in general. Let 
\begin{equation}
N_{A, c}=M_{A,c}^{\GFT}=
D(Y) / \left(\sum_{1 \leq i \leq n} 
D(Y) (Z_{i,c})^{\GFT} + \sum_{f \in I} D(Y) f \right) 
\end{equation} 
be the corresponding left $D(Y)$-module 
and $\N_{A, c}$ the coherent $\D_Y$-module 
associated to it. 
Now on the algebraic torus $T=(\CC^*)^n_x$ 
we define a holonomic $\D_T$-module $\R_c$ by 
\begin{equation}
\R_c= 
\D_T / \sum_{1 \leq i \leq n} \D_T 
\left\{ x_i \frac{\partial}{\partial x_i}+
(1-c_i) \right\} 
\simeq \sho_T x_1^{c_1-1} \cdots x_n^{c_n-1}.  
\end{equation} 
This is an integrable 
connection on $T$ and we have 
\begin{equation}
{\rm DR}_T (\R_c ) \simeq 
(\CC_{T^{\an}} 
x_1^{-c_1+1} \cdots x_n^{-c_n+1})[n].  
\end{equation} 
Let $v=[1] \in \N_{A,c}$ and $w_0 =[1] \in 
\R_c$ be the canonical generators. Recall that 
the transfer bimodule $\D_{T \longrightarrow Y}$ 
has the canonical section 
$1_{T \longrightarrow Y}$. We define 
a section $1_{Y \longleftarrow T}$ of 
$\D_{Y \longleftarrow T}= \Omega_T 
\otimes_{\sho_T} \D_{T \longrightarrow Y} 
\otimes_{j^{-1} \sho_Y} 
j^{-1} \Omega_Y^{\otimes -1}$ by 
\begin{equation}
1_{Y \longleftarrow T}
=(dx_1 \wedge \cdots \wedge dx_n) \otimes 
1_{T \longrightarrow Y} \otimes 
j^{-1}(d \zeta_1 \wedge \cdots \wedge 
 d \zeta_N)^{\otimes -1}. 
\end{equation} 
Note that this definition of 
$1_{Y \longleftarrow T}$ depends 
on the coordinates of $Y$ and $T$. 
Then we obtain a section $w$ of 
the regular holonomic $\D_Y$-module 
\begin{equation}
\CS_{A, c}:= \int_j \R_c = 
j_*(\D_{Y \longleftarrow T} \otimes_{\D_T} 
\R_c ) 
\end{equation} 
defined by $w=j_*(
1_{Y \longleftarrow T} \otimes w_0)$. 
We can easily check that this section 
$w \in \CS_{A,c}$ satisfies the system 
\eqref{DSY}. Hence as in 
\cite[page 268-269]{G-K-Z-2}, 
we obtain a morphism 
\begin{equation}
\Psi : \N_{A, c} \longrightarrow 
\CS_{A, c}= \int_j \R_c 
\end{equation} 
of left $\D_Y$-modules which sends 
the canonical generator $v=[1] \in \N_{A,c}$ 
to $w \in \CS_{A,c}$. 

\begin{definition}\label{NRC} 
(Gelfand-Kapranov-Zelevinsky 
\cite[page 262]{G-K-Z-2}) 
For a face $\Gamma$ of $\Delta$ 
containing the origin $0 \in \RR^n$ we 
denote by $\Lin (\Gamma ) \subset \CC^n$ 
the $\CC$-linear span of $\Gamma$. 
We say that the parameter vector 
$c \in \CC^n$ is nonresonant 
(with respect to $A$) if 
for any face $\Gamma$ of $\Delta$ 
of codimension $1$ such that 
$0 \in \Gamma$  
we have $c \notin \{ \ZZ^n+ 
\Lin (\Gamma ) \}$. 
\end{definition} 
Recall that if $c \in \CC^n$ is nonresonant 
then it is semi-nonresonant 
in the sense of \cite[page 284]{A}. 
The following result was proved by Saito \cite{S1} 
and Schulze-Walther \cite{SW1}, \cite{SW2} 
by using commutative algebras (see also 
Beukers \cite{B-0} for another approach 
to this problem). Here we 
give a geometric proof to it.  

\begin{lemma}\label{IRR} 
Assume that the parameter vector 
$c \in \CC^n$ is nonresonant. Then 
the regular holonomic $\D_Y$-module $\CS_{A,c}$ 
is irreducible. 
\end{lemma}
\begin{proof} 
Note that ${\rm DR}_T (\R_c ) \simeq 
(\CC_{T^{\an}} 
x_1^{-c_1+1} \cdots x_n^{-c_n+1})[n]$ is 
an irreducible perverse sheaf on 
$T^{\an}$. Then also its minimal extension by 
the locally closed embedding $j$ is 
irreducible (see 
\cite[Corollary 8.2.10]{H-T-T}). 
As in \cite[Theorem 3.5 and 
Propositions 3.2 and 4.4]{G-K-Z-2} it 
suffices to show that the canonical 
morphism 
\begin{equation}
j_!(\CC_{T^{\an}} x_1^{c_1-1} \cdots x_n^{c_n-1}) 
\longrightarrow 
Rj_* (\CC_{T^{\an}} x_1^{c_1-1} 
\cdots x_n^{c_n-1}) 
\end{equation} 
is a quasi-isomorphism. 
For this, we have only to prove the 
vanishing $Rj_* (\CC_{T^{\an}} x_1^{c_1-1} 
\cdots x_n^{c_n-1})_q \simeq 0$ 
for any $q \in \overline{j(T)} \setminus j(T)$. 
Note that by the nonresonance of $c \in \CC^n$ 
for any $p \in \ZZ$ and the local system $\LL 
 :=\CC_{T^{\an}} x_1^{c_1-1} 
\cdots x_n^{c_n-1}$ on $T^{\an}$ we have 
$H^p(T^{\an} ; \LL )=0$. Let $S(A) \subset \ZZ^n$ 
(resp. $K(A) \subset \RR^n$) be the 
semigroup (resp. the convex cone) generated 
by $A$. Then by (the proof of) 
\cite[Chapter 5, Theorem 2.3]{G-K-Z}  
we have $\overline{j(T)} \simeq \Spec (\CC [S(A)])$. 
Let us define an action of $T$ on $Y= \CC^N_{\zeta}$ by 
\begin{equation}
(\zeta_1, \ldots, \zeta_N) \longmapsto 
(x^{a(1)} \zeta_1, \ldots, x^{a(N)} \zeta_N) 
\end{equation}
for $x \in T$. Then by 
\cite[Chapter 5, Theorem 2.5]{G-K-Z} there 
exists a natural bijection 
between the faces of $K(A)$ and the $T$-orbits 
in $\overline{j(T)}$. In particular, if 
$K(A)= \RR^n$ we have $\overline{j(T)} =
j(T)$ and there is nothing to prove. 
First consider the 
case where $0 \in \RR^n$ is an apex of 
$K(A)$ and $q=0 \in Y=\CC^N_{\zeta}$. If 
$0 \in A$ i.e. $0=a(j)$ for some 
$1 \leq j \leq N$ we have $\overline{j(T)} 
\subset \{ \zeta_j=1 \} \simeq \CC^{N-1}$. 
Hence we may assume that $0 \notin A$ 
from the start. In this case, $\{ 0 \} 
\subset \overline{j(T)}$ is the unique 
$0$-dimensional $T$-orbit in 
$\overline{j(T)}$ which corresponds to 
$\{ 0 \} \prec K(A)$. From now on, we 
will prove that $Rj_*( \LL )_0 \simeq 0$. 
By our assumption there exists a 
linear function $l: \RR^n \longrightarrow 
\RR$ such that $l( \ZZ^n) \subset \ZZ$ 
and $K(A) \setminus \{ 0 \} \subset 
\{ l>0 \}$.  We define a real-valued function 
$\varphi : Y=\CC^N_{\zeta} \longrightarrow 
\RR$ by 
\begin{equation}
\varphi (\zeta )=|\zeta_1|^{\frac{C}{l(a(1))}} 
+ \cdots + |\zeta_N|^{\frac{C}{l(a(N))}}, 
\end{equation}
where we take $C \in \ZZ_{>0}$ large enough 
so that $\varphi$ and its level 
sets $\varphi^{-1} (b)$ ($b>0$) are smooth. Let 
$(l_1,l_2, \ldots, l_n) \in \ZZ^n$ be the 
coefficients of the linear function 
$l$. Define an action of the multiplicative 
group $\RR_{>0}$ on $T$ by 
\begin{equation}
r \cdot (x_1, \ldots, x_n)= 
(r^{l_1}x_1, \ldots, r^{l_n}x_n) 
\end{equation}
for $r \in \RR_{>0}$. Then we have 
\begin{equation}
j(r \cdot x)=(r^{l(a(1))}x^{a(1)}, \ldots, 
r^{l(a(N))}x^{a(N)}) 
\end{equation}
and hence 
\begin{equation}
\varphi (j(r \cdot x))= r^C \varphi (j(x)). 
\end{equation}
Therefore by the action of $\RR_{>0}$ on 
$Y= \CC^N_{\zeta}$ defined by 
\begin{equation}
r \cdot (\zeta_1, \ldots, \zeta_N)= 
(r^{l(a(1))} \zeta_1, \ldots, r^{l(a(N))} \zeta_N), 
\end{equation}
a level set $\varphi^{-1}(t)$ ($t>0$) of 
$\varphi$ is sent to the one 
$\varphi^{-1}(r^C t)$. Moreover this action 
preserves the $T$-orbits in $\overline{j(T)}$. 
Let $O \subset \overline{j(T)}$ be such a 
$T$-orbit. Then all the level sets 
$\varphi^{-1}(t)$ ($t>0$) of 
$\varphi$ are transversal to $O$, or 
all are not. But the latter case cannot 
occur by the Sard theorem. Then we obtain an 
isomorphism 
\begin{equation}
H^pRj_*(\LL )_0 \simeq H^p(\CC^N; Rj_*(\LL )) 
\simeq H^p(T^{\an} ; \LL ) \simeq 0
\end{equation}
for any $p \in \ZZ$. Next consider the remaining 
case where $q \in O$ for a $T$-orbit $O$ 
in $\overline{j(T)}$ such that 
$\dim O \geq 1$. Then in a neighborhood of 
$q$, the variety $\overline{j(T)}$ is a 
product $W \times O$ 
for an affine toric variety $W \subset 
\CC^{N^{\prime}}$ and 
$j(T)=(T_1 \sqcup \cdots \sqcup T_k) 
\times O$ for some tori $T_i \simeq 
(\CC^*)^{n- \dim O}$. See 
\cite[Chapter 5, Theorem 3.1]{G-K-Z} 
and the proof of 
\cite[Theorem 4.9]{M-T-1} 
for the details. Moreover for 
the semigroup $S(A_O) \subset \ZZ^{n- \dim O}$ 
generated by a finite subset $A_O \subset 
\ZZ^{n- \dim O}$ we have 
$\overline{T_i} \simeq \Spec (\CC [S(A_O)]) 
\subset W$ ($i=1,2, \ldots, k$). 
These varieties $\overline{T_i}$ 
are the irreducible components of 
$W$. For the explicit construction of 
$\overline{T_i}$ see the proof of 
\cite[Theorem 4.9]{M-T-1}. By this construction 
$0 \in \RR^{n- \dim O}$ is an apex 
of the convex cone $K(A_O) \subset 
\RR^{n- \dim O}$ generated by $A_O$. 
Let $p_2: W \times O 
\longrightarrow O$ and 
$q_2: T_i \times O 
\longrightarrow O$ be the second 
projections. Then it follows from 
the nonresonance of $c \in \CC^n$ the 
restriction of $\LL$ to $q_2^{-1}p_2 
(q) \simeq T_i$ is a non-constant 
local system. 
So we can apply our 
previous arguments and 
prove $Rj_*(\LL )_q \simeq 0$ in this 
case, too. This completes the proof. 
\end{proof} 

By Lemma \ref{IRR}, if 
$c \in \CC^n$ is nonresonant the 
non-trivial morphism $\Psi$ 
should be surjective. 
According to Schulze-Walther 
\cite[Corollary 3.7]{SW1} the 
morphism $\Psi$ is also an 
isomorphism in this case. 
Let $\FT :D(Y) \simto D(X)$ be the inverse 
of the Fourier transform $\GFT$. 
Then we have an isomorphism 
$N_{A, c}^{\FT} \simeq M_{A, c}$ 
of left $D(X)$-modules. The corresponding 
coherent $\D_X$-module $\N_{A, c}^{\FT} 
\simeq \M_{A, c}$ can be 
constructed also in the following way. 
Let $\sigma = \langle \cdot, \cdot  \rangle : X \times Y 
\longrightarrow \CC$ be the canonical pairing 
defined by $\langle z, \zeta \rangle =
\sum_{j=1}^N z_j \zeta_j$ and $p_1: X \times Y 
\longrightarrow X$ (resp. $p_2: X \times Y 
\longrightarrow Y$) the first (resp. second) 
projection. Then we have the following theorem 
due to Katz-Laumon \cite{K-L}. 

\begin{theorem}\label{KAL} 
(Katz-Laumon \cite{K-L}) 
For any $c \in \CC^n$ we have an isomorphism 
\begin{equation}
\N_{A, c}^{\FT} \simeq \int_{p_1} \left\{ 
(p_2^* \N_{A, c}) \otimes_{\sho_{X \times Y}} 
\sho_{X \times Y}e^{\sigma} \right\} , 
\end{equation}
where $\sho_{X \times Y}e^{\sigma} $ is 
the integrable connection associated to 
$e^{\sigma}: X \times Y 
\longrightarrow \CC$. 
\end{theorem}
In the same way, for any $c \in \CC^n$ we have 
\begin{equation}\label{KLS} 
\CS_{A, c}^{\FT} = \int_{p_1} \left\{ 
(p_2^* \CS_{A, c}) \otimes_{\sho_{X \times Y}} 
\sho_{X \times Y}e^{\sigma} \right\}.  
\end{equation} 
If moreover $c \in \CC^n$ 
is nonresonant, then by Lemma \ref{IRR} 
we obtain surjective morphisms 
$N_{A,c} \longrightarrow \CS_{A, c}(Y)$ 
and $\M_{A, c}(X) \simeq N_{A, c}^{\FT} 
\longrightarrow \CS_{A, c}^{\FT}(X)$. 
For nonresonant  $c \in \CC^n$, 
we thus obtain a surjective morphism 
\begin{equation}
\M_{A, c} \simeq \N_{A, c}^{\FT} 
\longrightarrow \CS_{A, c}^{\FT} 
\end{equation}
of left $\D_X$-modules. 
Let $e^{\tau}: X \times T \longrightarrow \CC$ 
be the function defined by $e^{\tau}(z,x)=
\exp (\sum_{j=1}^N z_j x^{a(j)})$ and 
$q_1: X \times T 
\longrightarrow X$ (resp. $q_2: X \times T 
\longrightarrow T$) the first (resp. second) 
projection. Then by the base change theorem 
\cite[Theorem 1.7.3 and Corollary 1.7.5]{H-T-T}, 
we have the isomorphism  
\begin{equation}
\CS_{A, c}^{\FT} \simeq \int_{q_1} \left\{ 
(q_2^* \R_c) \otimes_{\sho_{X \times T}} 
\sho_{X \times T}e^{\tau} \right\}.  
\end{equation}
Namely $\CS_{A, c}^{\FT}$ is the direct 
image of the integrable connection 
\begin{equation}
\CK =(q_2^* \R_c) \otimes_{\sho_{X \times T}} 
\sho_{X \times T}e^{\tau} 
\end{equation}
on $X \times T$ by $q_1$. Define 
a function $g: X \times T \longrightarrow \CC$ by 
\begin{equation}
g(z,x)=\exp (\sum_{j=1}^N z_j x^{a(j)}) 
x_1^{c_1-1} \cdots x_n^{c_n-1}.  
\end{equation}
Then by the results of Hien-Roucairol \cite{H-R} 
the holomorphic solution complex 
\begin{equation}
{\rm Sol}_X (\CS_{A, c}^{\FT})=\rhom_{\D_{X^{\an}}}
((\CS_{A, c}^{\FT})^{\an}, \sho_{X^{\an}})  
\end{equation}
of $\CS_{A, c}^{\FT}$ 
is expressed by the rapid decay homology groups 
associated the function $g$. 
Indeed, for $z \in \Omega$ let 
$\CK_z$ (resp. $g_z: T \longrightarrow \CC$)
be the restriction of the connection 
$\CK$ (resp. the function $g$) 
to $U_z:=q_1^{-1}(z) \simeq T \subset 
\Omega \times T$. Namely we set 
\begin{equation}
g_z(x)=\exp (\sum_{j=1}^N z_j x^{a(j)}) 
x_1^{c_1-1} \cdots x_n^{c_n-1}.  
\end{equation}
for $z \in U_z \simeq T$. 
Then $\CK_z \simeq \sho_{U_z}g_z$ and for 
the dual connection 
$\CK_z^* \simeq \sho_{U_z}( \frac{1}{g_z})$ 
of $\CK_z$ we have 
\begin{equation}
H^{-n}{\rm DR}_{T} (\CK_z^*) 
\simeq \CC_{U_z^{\an}}g_z. 
\end{equation}
Moreover for any $p \in \ZZ$, by Proposition 
\ref{RTH} (see also the proof of Theorem \ref{Main} 
below) the rapid decay homology groups 
\begin{equation}
H_p^{\rd}(U_z; \CK_z^* )
\qquad (z \in \Omega^{\an}) 
\end{equation}
associated 
to the integrable connections 
$\CK_z^*$ (or to the 
functions $g_z: T 
\longrightarrow \CC$) 
are isomorphic to each other 
and define a local system 
$\CH_p^{\rd}$ 
on $\Omega^{\an}$. See \cite{H-R} for the details. 
The following result is essentially due to 
Hien-Roucairol \cite{H-R}. 

\begin{theorem}\label{HAR} 
(Hien-Roucairol \cite{H-R}) For any 
$c \in \CC^n$ and $p \in \ZZ$ 
we have an isomorphism 
\begin{equation}
\CH_{n+p}^{\rd} \simeq 
H^p {\rm Sol}_X (\int_{q_1} \CK )
\simeq H^p {\rm Sol}_X (\CS_{A, c}^{\FT})
\end{equation}
of local systems on $\Omega^{\an}$. 
\end{theorem}
Recall that in \cite[Section 3]{A} Adolphson 
proved that $\M_{A, c}$ is an integrable 
connection on $\Omega$. 
From now on, we assume that $c \in \CC^n$ 
is nonresonant. Then we have the 
surjective morphism 
$\M_{A, c} \simeq \N_{A, c}^{\FT} 
\longrightarrow \CS_{A, c}^{\FT}$ 
and $\CS_{A, c}^{\FT}$ is also an integrable 
connection on $\Omega$. This in 
particular implies that 
for any $p \not= 0$ we have 
$H^p {\rm Sol}_X (\CS_{A, c}^{\FT}) 
\simeq 0$. Hence we get 
\begin{equation}\label{key} 
H_{n+p}^{\rd}(U_z; \CK_z^*) \simeq 0 
\qquad (p \not= 0, \quad z \in \Omega^{\an}). 
\end{equation}
It follows also from the 
surjection 
$\M_{A, c} \longrightarrow 
\CS_{A, c}^{\FT}$ that we have an injection 
\begin{equation}\label{EQI} 
\Phi : \CH_{n}^{\rd}  \simeq 
\shom_{\D_{X^{\an}}}
((\CS_{A, c}^{\FT})^{\an}, \sho_{X^{\an}})
 \hookrightarrow \shom_{\D_{X^{\an}}}
(( \M_{A, c} )^{\an}, \sho_{X^{\an}}). 
\end{equation}
By using the generator 
\begin{equation}
u=[1] \in \M_{A, c}=
\D_{X} / \left(\sum_{1 \leq i \leq n} 
\D_{X} Z_{i,c} + \sum_{\mu \in {\rm Ker} A 
\cap \ZZ^{N} } \D_{X} \square_{\mu} \right) 
\end{equation}
of $\M_{A, c}$ we regard $\shom_{\D_{X^{\an}}}
(( \M_{A, c} )^{\an}, \sho_{X^{\an}})$ as 
a subsheaf of $\sho_{X^{\an}}$. Then 
we have the following result. 

\begin{theorem}\label{Main} 
Assume that the parameter vector 
$c \in \CC^n$ is nonresonant. 
Then the morphism 
$\Phi$ induces an isomorphism 
\begin{equation}\label{ETT} 
\CH_{n}^{\rd} \simeq 
\shom_{\D_{X^{\an}}}
(( \M_{A, c} )^{\an}, \sho_{X^{\an}}) 
\end{equation}
of local systems on $\Omega^{\an}$. 
Moreover this isomorphism is given by 
the integral 
\begin{equation}
\gamma \longmapsto \left\{ \Omega^{\an} \ni z 
\longmapsto \int_{\gamma^z} 
\exp (\sum_{j=1}^N z_j x^{a(j)}) 
x_1^{c_1-1} \cdots x_n^{c_n-1} 
dx_1 \wedge \cdots \wedge dx_n \right\}, 
\end{equation}
where for a continuous family $\gamma$ of rapid 
decay $n$-cycles in $\Omega^{\an} \times T^{\an}$ 
and $z \in \Omega^{\an}$ we 
denote by $\gamma^z$ its restriction 
$\gamma \cap U_z$ to $U_z=q_1^{-1}(z) \simeq T$. 
\end{theorem} 
Note that this integral representation of 
the confluent $A$-hypergeometric functions 
$\shom_{\D_{X^{\an}}}
(( \M_{A, c} )^{\an}, \sho_{X^{\an}})$ 
coincides with the one in Adolphson 
\cite[Equation (2.6)]{A}. 
\begin{proof}
Recall that the sheaf $\shom_{\D_{X^{\an}}}
(( \M_{A, c} )^{\an}, \sho_{X^{\an}}) $ is 
a local system on $\Omega^{\an}$. Moreover 
by \cite[Corollary 5.20]{A} its rank 
is $\Vol_{\ZZ}(\Delta )$. So it 
suffices to show that for any 
$z \in \Omega^{\an}$ the dimension of 
the $n$-th rapid decay homology group 
$H_n^{\rd}(U_z; \CK_z^*)$ is also 
$\Vol_{\ZZ}(\Delta )$. Let 
\begin{equation}
\Eu^{\rd}(U_z; \CK_z^*) 
:=\sum_{p \in \ZZ} (-1)^p 
\dim H_p^{\rd}(U_z; \CK_z^*)
\end{equation}
be the rapid decay Euler characteristic. 
Then by \eqref{key} we have only to 
prove the equality 
\begin{equation}\label{AIM} 
\Eu^{\rd}(U_z; \CK_z^*)
=(-1)^n \Vol_{\ZZ}(\Delta ). 
\end{equation}
Let $\Sigma_0$ be the dual fan of 
$\Delta = {\rm conv} (A \cup \{ 0 \} )$ in $\RR^n$ 
and $\Sigma$ its smooth subdivision. 
Denote by $Z_{\Sigma}$ the smooth toric 
variety associated to the fan $\Sigma$. 
Then $Z_{\Sigma}$ is a smooth 
compactification of $U_z \simeq T$ such 
that $Z_{\Sigma} \setminus U_z$ is 
a normal crossing divisor. 
However on $Z_{\Sigma}$ there still 
remain some points where the zero and the 
pole of the meromorphic extension of 
$h_z(x)= \sum_{j=1}^N z_j x^{a(j)}$ to it 
meet. We call them the points 
of indeterminacy of $h_z$. 
By using the 
non-degeneracy of the Laurent polynomial 
$h_z(x)$, as in \cite[Section 3]{M-T-3} we then 
construct a complex blow-up $Z:=\tl{Z_{\Sigma}}$ 
of $Z_{\Sigma}$ such that the meromorphic 
extension of $h_z$ to it has no point 
of indeterminacy. For the reader's convenience, 
we briefly recall the construction of $Z$. Recall that 
$T$ acts on $Z_{\Sigma}$ and the $T$-orbits are parametrized 
by the cones in $\Sigma$. For a cone $\sigma \in \Sigma$ 
we denote by $T_{\sigma} \simeq (\CC^*)^{n- \dim \sigma}$ 
the corresponding $T$-orbit. Let $\rho_1, \ldots, \rho_m 
\in \Sigma$ be the rays i.e. the one-dimensional cones in 
$\Sigma$. By using the primitive vectors $\kappa_i \in 
\rho_i \cap (\ZZ^n \setminus \{ 0 \} )$ on $\rho_i$ 
we set 
\begin{equation}
m_i= - \min_{a \in \Delta} \langle \kappa_i, a 
\rangle \geq 0. 
\end{equation}
We renumber  $\rho_1, \ldots, \rho_m$ 
so that $m_i>0$ if and only if 
$1 \leq i \leq l$ for some $1 \leq l \leq m$. Then for 
any $1 \leq i \leq l$ the meromorphic extension of 
$h_z$ to $Z_{\Sigma}$ has a pole of order $m_i >0$ 
along the toric divisor $D_i= \overline{T_{\rho_i}} 
\subset Z_{\Sigma}$. By the non-degeneracy of $h_z$ 
the hypersurface $\overline{h_z^{-1}(0)} \subset 
Z_{\Sigma}$ intersects $D_I= \cap_{i \in I} D_i$ 
transversally 
for any subset $I \subset \{ 1,2, \ldots, m \}$ 
such that $I \cap \{ 1,2, \ldots, l \} 
\not= \emptyset$ (see Definition \ref{AND}). 
The meromorphic extension of 
$h_z$ has points of indeterminacy in $\cup_{i=1}^l 
( \overline{h_z^{-1}(0)} \cap D_i)$.

\begin{center}
\includegraphics[height=8cm]{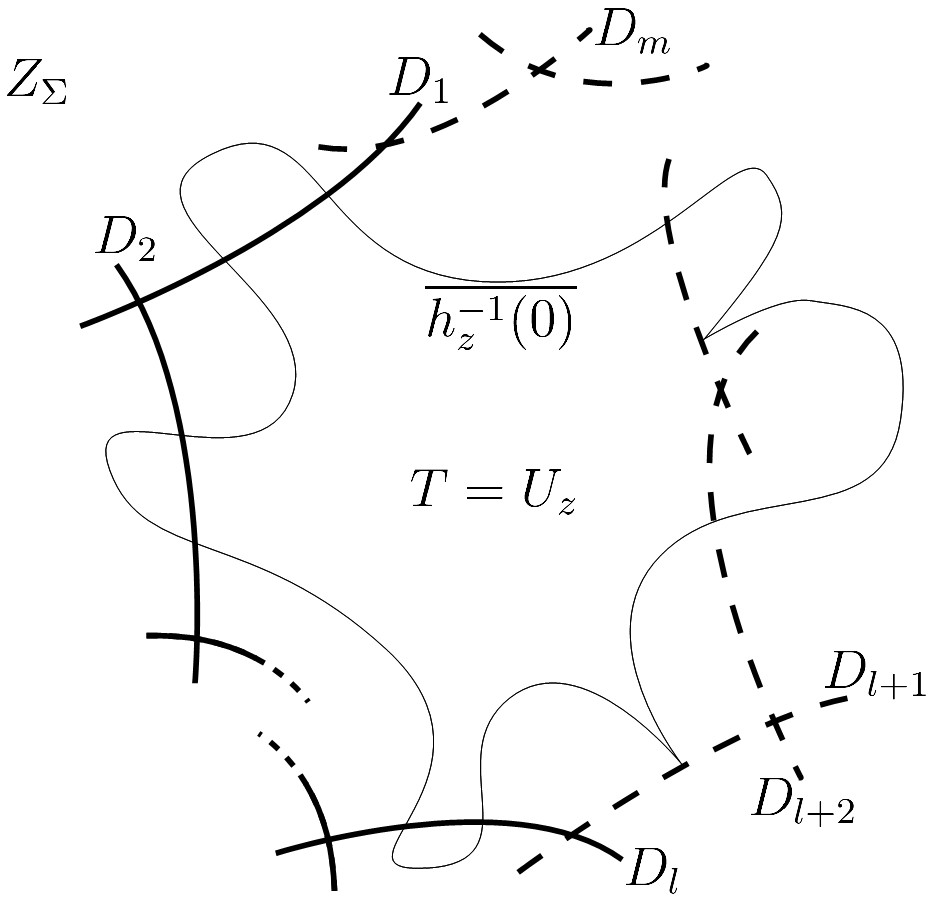}

Figure 2.
\end{center}


First we construct a tower of 
$m_1$ codimension-two blow-ups 
over $\overline{h_z^{-1}(0)} \cap D_1$ 
(see \cite[Section 3]{M-T-3} 
and \cite[Section 3 and Lemma 4.9]{M-T-4} 
for the details). 
Then the indeterminacy of $h_z$ 
over $D_1 \setminus ( \cup_{j \not= 1} D_j)$ 
is eliminated. 
By repeating this construction also over 
(the proper transforms of) 
$D_2, D_3, \ldots, D_l$ we finally obtain the desired 
proper morphism $Z= \tilde{Z_{\Sigma}} 
\longrightarrow Z_{\Sigma}$ 
of $Z_{\Sigma}$ as the figure below.

\begin{center}
\includegraphics[height=7cm]{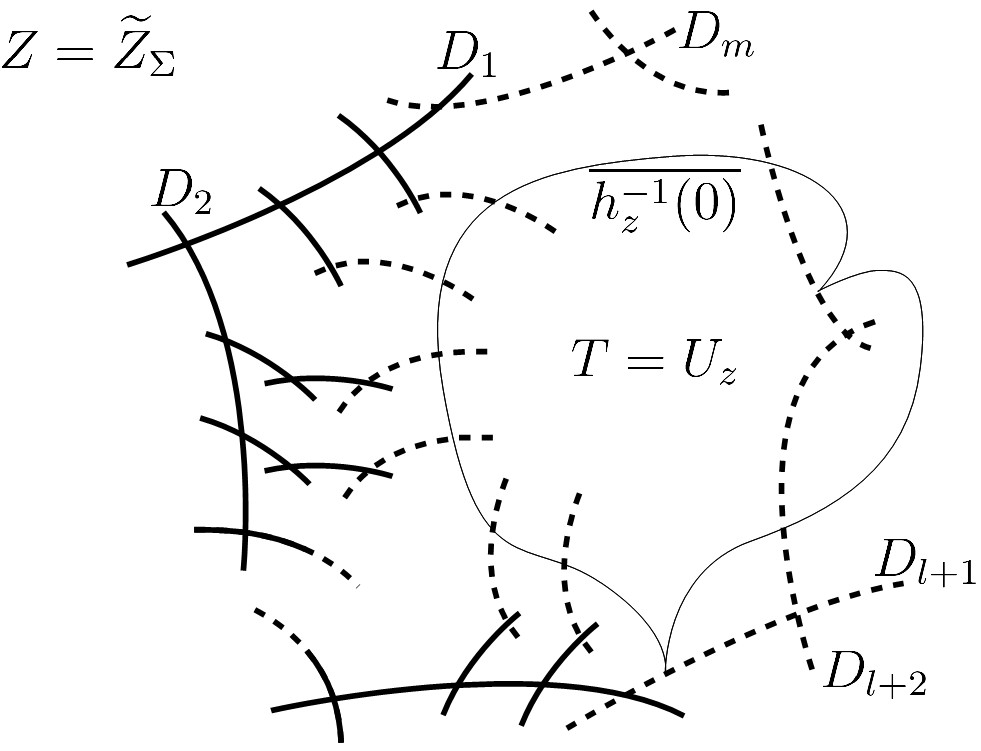}

Figure 3.
\end{center}


Now we can use this smooth 
compactification $Z$ of $U_z \simeq 
T$ and the normal crossing divisor 
$D:=Z \setminus U_z$ in it to define the 
rapid decay homology groups associated to the 
function $g_z(x)=\exp (h_z(x)) 
x_1^{c_1-1} \cdots x_n^{c_n-1}$ 
(see \cite[Section 2.1]{H-R}).  
In the figure of $Z$ above, 
the dotted curves stand for the irrelevant 
components (see the last half of 
Section \ref{sec:3}) 
of $D$. Note that 
for any $1 \leq i \leq l$ the inverse image 
of $\overline{h_z^{-1}(0)} \cap D_i$ 
by the morphism $Z= \tl{Z_{\Sigma}} 
\longrightarrow Z_{\Sigma}$ is a union 
of $\PP^1$-bundles on it and only the 
last one among them is irrelevant. 
Let $\pi : \tilde{Z} 
\longrightarrow Z^{\an}$ be the real oriented 
blow-up of $Z^{\an}$ along $D^{\an}$ and set 
$\tilde{D} =\pi^{-1}(D^{\an})$. 
Then, as in Section \ref{sec:3} we define 
the rapid decay homology groups 
$H_p^{\rd}(U_z; \CK_z^*)$ by using 
$\pi : \tilde{Z} 
\longrightarrow Z^{\an}$, 
$\tilde{D}$ and $g_z$ etc. By 
defining the set $Q \subset \tilde{D}$ 
of the rapid decay directions of 
$g_z(x)$ etc. as in Proposition \ref{RTH}, 
for the local system 
$\LL =
\CC_{T^{\an}}x_1^{c_1-1} \cdots x_n^{c_n-1}$ 
on $(U_z)^{\an} \simeq T^{\an}$ 
and the inclusion map 
$\iota: (U_z)^{\an} \simeq T^{\an} 
\hookrightarrow \tilde{Z}$ 
we obtain isomorphisms 
\begin{equation}
H_p^{\rd}(U_z ; \CK_z^*) \simeq 
H_p(T^{\an} \cup Q, Q; \iota_*(\LL ) ) 
\qquad (p \in \ZZ ). 
\end{equation}
For $1 \leq i \leq l$ we define a face 
$\Gamma_i \prec \Delta$ of $\Delta$ by 
\begin{equation}
\Gamma_i= \{ b \in \Delta \ | \ 
\langle \kappa_i, b \rangle = \min_{a \in \Delta} 
\langle \kappa_i, a \rangle \}. 
\end{equation}
We call it the supporting face of 
$\rho_i$ in $\Delta$. Denote by 
$v_i \geq 0$ the normalized 
(or simplicial) $(n-1)$-dimensional 
volume $\Vol_{\ZZ}( \Gamma_i) 
\in \ZZ_+$ of $\Gamma_i$. Then the $\Gamma_i$-part  
$h_z^{\Gamma_i}$ of $h_z$ is naturally 
identified with the defining (Laurent) 
polynomial of the hypersurface 
$T_{\rho_i} \cap \overline{h_z^{-1}(0)}$ 
in $T_{\rho_i} \simeq (\CC^*)^{n-1}$. 
Moreover by the Bernstein-Khovanskii-Kouchnirenko 
theorem (see \cite{Khovanskii}) its 
Euler characteristic is equal to 
$(-1)^n v_i$. Note that we have 
$\sum_{i=1}^l (v_i \times m_i)= 
\Vol_{\ZZ}( \Delta )$. 
Indeed, for the convex hulls $\widehat{\Gamma_i}$ 
of $\Gamma_i \cup \{ 0 \}$ in $\RR^n$ we 
have $\cup_{i=1}^l \widehat{\Gamma_i} = \Delta$ 
and $\sum_{i=1}^l 
\Vol_{\ZZ}( \widehat{\Gamma_i} )= 
\Vol_{\ZZ}( \Delta )$. Since $\Vol_{\ZZ}( 
\widehat{\Gamma_i} )$ 
(resp. $v_i= \Vol_{\ZZ}( \Gamma_i )$) 
is $n!$ times (resp. $(n-1)!$ times)  
the Euclidean volume of $\widehat{\Gamma_i}$ 
(resp $\Gamma_i$) and $m_i>0$ is the 
lattice height of $\widehat{\Gamma_i}$ from 
its base $\Gamma_i \prec \widehat{\Gamma_i}$, 
we have also $\Vol_{\ZZ}( \widehat{\Gamma_i} ) 
= \Vol_{\ZZ}( \Gamma_i ) \times m_i= 
v_i \times m_i$ for any $1 \leq i \leq l$. 

\medskip \par 
Let $Z= \sqcup_{\alpha} Z_{\alpha}$ be the 
canonical stratification of 
$Z= \tl{Z_{\Sigma}}$ associated 
to the normal crossing divisor $D=Z \setminus 
T$ and $E \subset Z$ the 
union of the exceptional divisors of the blow-up 
$Z= \tl{Z_{\Sigma}} 
\longrightarrow Z_{\Sigma}$. Then for any 
$1 \leq i \leq l$ there exists a unique stratum 
$Z_{\alpha_i}$ such that $T_{\rho_i} \setminus E= 
Z_{\alpha_i}$. For each stratum $Z_{\alpha}$ in 
the stratification 
we take its sufficiently small tubular 
neighborhood $V_{\alpha}$ in $Z$. For 
$1 \leq i \leq l$ we denote the 
alternating sum 
\begin{equation}\label{Eul} 
\sum_{p \in \ZZ} (-1)^p  
H_p( (V_{\alpha_i} \cap T^{\an}) 
\cup ( \pi^{-1}( V_{\alpha_i} ) \cap Q) , 
( \pi^{-1}( V_{\alpha_i} ) \cap Q) ; 
\iota_*( \LL ) ) 
\end{equation}
simply by $\Eu^{\rd}_i$. 
Then by applying Lemmas \ref{EC} and \ref{ECL} 
to the Mayer-Vietoris exact 
sequences for the relative twisted homology groups 
\begin{equation} 
H_p( (V_{\alpha} \cap T^{\an}) 
\cup ( \pi^{-1}( V_{\alpha} ) \cap Q) , 
( \pi^{-1}( V_{\alpha} ) \cap Q) ; 
\iota_*( \LL ) ) 
\end{equation}
and the geometric situation in Figure $3$ above, 
we can easily show that
\begin{equation} 
\Eu^{\rd}(U_z; \CK_z^*)= 
\sum_{p \in \ZZ} (-1)^p \dim 
H_p(T^{\an} \cup Q, Q; \iota_*(\LL ) )
= \sum_{i=1}^l \Eu^{\rd}_i. 
\end{equation}
Moreover by the proof of Lemma \ref{EC} and 
and Lemma \ref{ECL}, for any $1 \leq i \leq l$ 
we have $\Eu^{\rd}_i= (-1)^n v_i \times m_i$. 
Then the equality \eqref{AIM} follows 
from $\sum_{i=1}^l (v_i \times m_i)= 
\Vol_{\ZZ}( \Delta )$. This completes 
the proof of the isomorphism \eqref{ETT}. 

\medskip \par 
Let us prove the remaining assertion. 
Denote the distinguished section 
$(q_2^{*}w_0) \otimes e^{\tau}$ 
of the integrable connection 
$\CK =(q_2^* \R_c) \otimes_{\sho_{X \times T}} 
\sho_{X \times T}e^{\tau}$ by $t$. 
Let $\Omega^{\bullet}_{X \times T/X} 
\otimes_{\sho_{X \times T}} \CK$ be the 
relative algebraic de Rham complex of 
$\CK$ associated to the morphism 
$q_1 : X \times T \longrightarrow X$. 
Then we have an isomorphism 
\begin{equation}
\CS_{A, c}^{\FT} \simeq \int_{q_1} \CK 
\simeq H^n \left\{ 
(q_1)_*(\Omega^{\bullet}_{X \times T/X} 
\otimes_{\sho_{X \times T}} \CK ) 
\right\}. 
\end{equation}
For a relative $n$-form $\omega \in (q_1)_* 
\Omega^{n}_{X \times T/X}$ denote by 
$\cl (\omega \otimes t)$ the section of 
$\CS_{A, c}^{\FT}$ which corresponds to 
the cohomology class $[(q_1)_*(\omega 
\otimes t)] \in H^n \{ 
(q_1)_*(\Omega^{\bullet}_{X \times T/X} 
\otimes_{\sho_{X \times T}} \CK ) \}$ 
by the above isomorphism. According to the 
result of \cite{H-R}, by the isomorphism 
\begin{equation}
\CH_{n}^{\rd} \simeq 
\shom_{\D_{X^{\an}}}
(( \CS_{A, c}^{\FT} )^{\an}, \sho_{X^{\an}}) 
\end{equation}
of local systems 
on $\Omega^{\an}$, 
a family of rapid decay cycles 
$\gamma \in \CH_{n}^{\rd}$ is 
sent to the section 
\begin{equation}
\left[ ( \CS_{A, c}^{\FT} )^{\an} \ni 
f \otimes \cl (\omega \otimes t) 
\longmapsto \left\{ \Omega^{\an} \ni z 
\longmapsto f(z) \int_{\gamma^z} 
\exp (\sum_{j=1}^N z_j x^{a(j)}) 
x_1^{c_1-1} \cdots x_n^{c_n-1} 
\omega \right\} 
\right] 
\end{equation} 
($f \in \sho_{X^{\an}}$) 
of $\shom_{\D_{X^{\an}}}
(( \CS_{A, c}^{\FT} )^{\an}, \sho_{X^{\an}})$. 
Then the remaining assertion 
follows from the lemma below. 
This completes the proof. 
\end{proof}

\begin{remark}\label{TRAN} 
When $0 \in \Int ( \Delta )$ the irrelevant 
components of $D$ in the proof above are 
the last $\PP^1$-bundles 
on $\overline{h_z^{-1}(0)} \cap D_i$ 
($1 \leq i \leq l=m$). By the construction of 
the morphism $Z= \tilde{Z_{\Sigma}} 
\longrightarrow Z_{\Sigma}$ we can easily 
see that for any $t \in \CC$ the hypersurface 
$\overline{h_z^{-1}(t)} \subset Z$ intersects 
them transversally. 
\end{remark} 

\begin{lemma}  
By the morphism 
\begin{equation}
\M_{A,c} \longrightarrow 
\CS_{A, c}^{\FT} \simeq H^n \left\{ 
(q_1)_*(\Omega^{\bullet}_{X \times T/X} 
\otimes_{\sho_{X \times T}} \CK ) 
\right\}
\end{equation}
the canonical section $u=[1] \in \M_{A,c}$ 
is sent to the cohomology class 
$\cl ( (dx_1 \wedge \cdots 
\wedge dx_n) \otimes t)$. 
\end{lemma} 
\begin{proof} 
First note that the morphism $\Psi^{\FT}(X): 
\M_{A,c}(X) \simeq M_{A,c} \simeq 
N_{A,c}^{\FT} \longrightarrow 
\CS_{A,c}^{\FT}(X) \simeq \CS_{A,c}(Y)$ 
sends the canonical generator $u=[1] 
\in \M_{A,c}(X)$ to $w= 
j_*(1_{Y \longleftarrow T} 
\otimes w_0) \in \CS_{A,c}(Y)$. 
On the other hand, by \eqref{KLS} 
we have an isomorphism 
\begin{equation} 
\CS_{A, c}^{\FT} \simeq H^N \left[ (p_1)_* 
\left\{ \Omega^{\bullet}_{X \times Y/X} 
\otimes_{\sho_{X \times Y}} 
(p_2^* \CS_{A, c}) \otimes_{\sho_{X \times Y}} 
\sho_{X \times Y}e^{\sigma} \right\} \right].  
\end{equation} 
Then by Malgrange's simple proof 
\cite[page 135]{Malgrange} of Theorem 
\ref{KAL}, via this isomorphism the section 
$w \in \CS_{A,c}^{\FT}(X) \simeq \CS_{A,c}(Y)$ 
corresponds to the cohomology class 
\begin{equation} 
\left[ (p_1)_* 
\left\{ (d \zeta_1 \wedge \cdots \wedge d \zeta_N) 
\otimes (p_2^* w) \otimes 
e^{\sigma} \right\} \right].  
\end{equation} 
Let $\tilde{j}: X \times T \hookrightarrow 
X \times Y$ be the embedding induced by $j$. 
By the isomorphism 
\begin{equation} 
\CS_{A, c}^{\FT} \simeq H^N \left[ (p_1)_* 
\left\{ \Omega^{\bullet}_{X \times Y/X} 
\otimes_{\sho_{X \times Y}} 
\tilde{j}_* \left( 
\D_{X \times Y \longleftarrow X \times T} 
\otimes_{X \times T} \CK \right) 
\right\} \right] 
\end{equation} 
the above cohomology class corresponds to 
the one 
\begin{equation} 
\rho := \left[ (p_1)_* 
\left\{ (d \zeta_1 \wedge \cdots \wedge d \zeta_N) 
\otimes \tilde{j}_*(1_{X \times Y \longleftarrow 
X \times T} \otimes t) \right\} \right], 
\end{equation} 
where the section $1_{X \times Y \longleftarrow 
X \times T} \in \D_{X \times Y \longleftarrow 
X \times T}$ is defined similarly to 
$1_{Y \longleftarrow T} \in 
\D_{Y \longleftarrow T}$. Then it suffices 
to show that via the isomorphism 
\begin{equation} 
\CS_{A, c}^{\FT} \simeq 
\int_{p_1} \int_{\tilde{j}} \CK 
\simeq \int_{q_1} \CK 
\end{equation} 
the cohomology class $\rho$ is 
sent to $\cl ( (dx_1 \wedge 
\cdots \wedge dx_n) \otimes t)= 
[(q_1)_* \{  (dx_1 \wedge 
\cdots \wedge dx_n) \otimes t \} ]$ in 
\begin{equation}
\int_{q_1} \CK 
\simeq H^n \left\{ 
(q_1)_*(\Omega^{\bullet}_{X \times T/X} 
\otimes_{\sho_{X \times T}} \CK ) 
\right\}. 
\end{equation}
Since $X$ and $X \times Y$ are affine, 
we have only to prove that via the 
isomorphism 
\begin{equation}
H^N \Gamma (X \times Y; 
\Omega^{\bullet}_{X \times Y/X} 
\otimes_{\sho_{X \times Y}} 
\int_{\tilde{j}} \CK ) \simeq  
H^n \Gamma (X \times T; 
\Omega^{\bullet}_{X \times T/X} 
\otimes_{\sho_{X \times T}} \CK )
\end{equation}
the cohomology class 
\begin{equation} 
\left[ (d \zeta_1 \wedge \cdots \wedge d \zeta_N) 
\otimes \tilde{j}_*(1_{X \times Y \longleftarrow 
X \times T} \otimes t) \right] 
\end{equation} 
is sent to  
$ [ (dx_1 \wedge 
\cdots \wedge dx_n) \otimes t ]$. 
Indeed, we have isomorphisms 
\begin{eqnarray}
&  & 
H^0 \Gamma (X \times Y; 
\Omega^{N+ \bullet}_{X \times Y/X} 
\otimes_{\sho_{X \times Y}} 
\int_{\tilde{j}} \CK )
\\
 & \simeq & 
H^0 \Gamma (X \times Y; 
p_2^{-1} \Omega^{N+ \bullet}_{Y} 
\otimes_{p_2^{-1} \sho_{Y}} \tilde{j}_* 
(q_2^{-1} \D_{Y \longleftarrow T} 
\otimes_{q_2^{-1} \D_{T}} \CK) )
\\
 & \simeq & 
H^0 \Gamma (X \times Y; \tilde{j}_* \{ 
\tilde{j}^{-1} p_2^{-1} \Omega^{N+ \bullet}_{Y} 
\otimes_{\tilde{j}^{-1} p_2^{-1} \sho_{Y}}  
(q_2^{-1} \D_{Y \longleftarrow T} 
\otimes_{q_2^{-1} \D_{T}} \CK) \} )
\\
 & \simeq & 
H^0 \Gamma (X \times T; 
q_2^{-1}(j^{-1} \Omega^{N+ \bullet}_{Y} 
\otimes_{j^{-1} \sho_{Y}} \D_{Y \longleftarrow T}) 
\otimes_{q_2^{-1} \D_{T}} \CK )
\end{eqnarray}
by which the element 
$[(d \zeta_1 \wedge \cdots \wedge d \zeta_N) 
\otimes \tilde{j}_*(1_{X \times Y \longleftarrow 
X \times T} \otimes t)] $ is sent to the one 
$[q_2^{-1} \{ j^{-1} 
( d \zeta_1 \wedge \cdots \wedge d \zeta_N )
\otimes 1_{Y \longleftarrow T} \} \otimes t]$. 
Let $\CP^{\bullet} \simto \CK$ be a free 
resolution of the left $\D_{X \times T}$-module 
$\CK$. Since $X \times T$ is affine, 
we obtain a surjective homomorphism 
\begin{equation} 
\Gamma (X \times T; \CP^{0}) 
\longrightarrow 
\Gamma (X \times T; \CK ) 
\end{equation} 
and can take a lift $\hat{t} \in 
\Gamma (X \times T; \CP^{0})$ 
of $t \in \Gamma (X \times T; \CK )$. 
Moreover by the flatness of the right 
$\D_T$-module $\D_{Y \longleftarrow T}$  
and the well-known formula 
\begin{equation} 
j^{-1} \Omega^{N+ \bullet}_{Y} 
\otimes_{j^{-1} \sho_{Y}} \D_{Y \longleftarrow T} 
\simeq 
j^{-1} \Omega^{N}_{Y} 
\otimes^L_{j^{-1} \D_{Y}} \D_{Y \longleftarrow T}
\simeq \Omega_T^n 
\end{equation} 
there exists an isomorphism 
\begin{eqnarray}
 &  
H^0 \Gamma (X \times T; 
q_2^{-1}(j^{-1} \Omega^{N+ \bullet}_{Y} 
\otimes_{j^{-1} \sho_{Y}} \D_{Y \longleftarrow T}) 
\otimes_{q_2^{-1} \D_{T}} \CK )
 & 
\\ 
& \simeq 
H^0 \Gamma (X \times T; 
q_2^{-1}(j^{-1} \Omega^{N+ \bullet}_{Y} 
\otimes_{j^{-1} \sho_{Y}} \D_{Y \longleftarrow T}) 
\otimes_{q_2^{-1} \D_{T}} \CP^{\bullet} )
 & 
\\
& \simeq 
H^0 \Gamma (X \times T; 
q_2^{-1} \Omega_T^n  
\otimes_{q_2^{-1} \D_{T}} \CP^{\bullet} )
 & 
\end{eqnarray}
by which $[q_2^{-1} \{ j^{-1} 
( d \zeta_1 \wedge \cdots \wedge d \zeta_N )
\otimes 1_{Y \longleftarrow T} \} \otimes t]$ 
is sent to $[q_2^{-1} (dx_1 \wedge 
\cdots \wedge dx_n) \otimes \hat{t}]$. 
Similarly, by the isomorphism 
\begin{equation}
H^0 \Gamma (X \times T; 
q_2^{-1} \Omega_T^n  
\otimes_{q_2^{-1} \D_{T}} \CP^{\bullet} )
\simeq  
H^0 \Gamma (X \times T; 
\Omega^{n+ \bullet}_{X \times T/X} 
\otimes_{\sho_{X \times T}} \CK ) 
\end{equation}
the element $[q_2^{-1} (dx_1 \wedge 
\cdots \wedge dx_n) \otimes \hat{t}]$ is sent to 
$[(dx_1 \wedge \cdots \wedge dx_n) \otimes t]$. 
This completes the proof. 
\end{proof} 

As a corollary of Theorem \ref{Main}, 
we recover the following Saito and Schulze-Walther's 
construction of Adolphson's 
confluent $A$-hypergeometric $\D$-module 
$\M_{A,c}$ on $\Omega \subset X=\CC^A$. 

\begin{corollary} 
(Saito \cite{S1} and Schulze-Walther 
\cite{SW1}, \cite{SW2}) 
Assume that the parameter vector 
$c \in \CC^n$ is nonresonant. Then 
we have an isomorphism $\M_{A,c} \simto 
\CS_{A,c}^{\FT}$ of integrable connections 
on $\Omega$. In particular, 
$\M_{A,c}$ is an irreducible 
connection there. 
\end{corollary}
This result was first obtained in Saito \cite{S1} 
and Schulze-Walther \cite{SW1}, \cite{SW2} 
by using totally different methods. 
In fact, Saito \cite{S1} proved moreover that 
we have an isomorphism $\M_{A,c} \simto 
\CS_{A,c}^{\FT}$ on the whole $X$. 

\begin{remark}
Since $\CS_{A, c}$ 
is regular holonomic 
by a theorem of Hotta \cite{Hotta}, 
it is also regular at infinity in 
the sense of Daia \cite{Daia}. Then 
by using the Fourier-Sato transforms 
(see \cite{Malgrange}) 
we can apply the main theorem of 
Daia \cite{Daia} to get another 
sheaf-theoretical (or functorial) 
construction of the sheaf 
$\CH_{n}^{\rd}  \simeq 
\shom_{\D_{X^{\an}}} 
((\CS_{A, c}^{\FT})^{\an}, \sho_{X^{\an}})$. 
This construction is valid  
even when the parameter $c \in \CC^n$ is 
resonant. However if $c \in \CC^n$ 
is resonant, we cannot prove that the 
morphism \eqref{EQI} is an isomorphism. 
Namely for such $c \in \CC^n$, 
the sheaf $\CH_{n}^{\rd}  \simeq 
\shom_{\D_{X^{\an}}} 
((\CS_{A, c}^{\FT})^{\an}, \sho_{X^{\an}})$ 
may be different from the one 
$\shom_{\D_{X^{\an}}} 
((\M_{A, c})^{\an}, \sho_{X^{\an}})$ 
of confluent $A$-hypergeometric functions. 
\end{remark} 

\begin{example}
Assume that $n=1$ and $T=\CC^*_x$. 
\par \noindent (i) If $A= \{ 1,-1 \} \subset \ZZ$ 
our integral representation of the $A$-hypergeometric 
functions $u(z_1,z_2)$ on $\CC^2_z$ is 
\begin{equation}
u(z_1,z_2)= \int_{\gamma^z} 
\exp (z_1x + \frac{z_2}{x}) 
x^{c-1} dx. 
\end{equation}
Restricting the function $u(z_1,z_2)$ 
to $\CC_t$ by the inclusion 
map $\CC_t \hookrightarrow \CC^2_z$, 
$t \longmapsto (\frac{t}{2}, -\frac{t}{2})$ we obtain 
the classical Bessel function 
\begin{equation}
v(t)= \frac{1}{2 \pi i} 
\int_{\gamma^{(\frac{t}{2}, -\frac{t}{2})}} 
\exp (\frac{tx}{2} - \frac{t}{2x}) 
x^{- \nu -1} dx
\end{equation}
for the parameter $\nu =-c$. 
Here $\gamma^{(\frac{t}{2}, -\frac{t}{2})} 
\subset \CC$ is the path which comes from 
infinity along the line ${\rm arg} x= - \pi$, 
turns around the origin and goes back to 
infinity along ${\rm arg} x= \pi$. 
\par \noindent (ii) If $A= \{ 3, 1 \} \subset \ZZ$ 
our integral representation of the $A$-hypergeometric 
functions $u(z_1,z_2)$ on $\CC^2_z$ is 
\begin{equation}
u(z_1,z_2)= \int_{\gamma^z} 
\exp (z_1x^3 + z_2x) 
x^{c-1} dx. 
\end{equation}
Restricting the function $u(z_1,z_2)$ 
to $\CC_t$ by the inclusion 
map $\CC_t \hookrightarrow \CC^2_z$, 
$t \longmapsto (\frac{1}{3}, -t)$ we obtain 
the classical Airy function 
\begin{equation}
v(t)= \frac{1}{2 \pi i} 
\int_{\gamma^{(\frac{1}{3}, -t)}} 
\exp (\frac{x^3}{3} - tx) dx
\end{equation}
for $c=1$. 
Here $\gamma^{(\frac{1}{3}, -t)} 
\subset \CC$ is the path which comes from 
infinity along the line ${\rm arg} 
x= - \frac{\pi}{3}$ and goes back to 
infinity along ${\rm arg} x= \frac{\pi}{3}$. 
\end{example}

\section{Asymptotic expansions at infinity of 
confluent $A$-hypergeometric functions}\label{sec:5}

In this section, assuming the condition 
$0 \in \Int (\Delta )$ 
we construct natural bases of the rapid 
decay homology groups $(\CH_{n}^{\rd})_z \simeq 
H_n^{\rd}(U_z; \CK_z^* )$ and apply them to prove a 
formula for the asymptotic expansions 
at infinity of Adolphson's 
confluent $A$-hypergeometric functions.

\subsection{Preliminary results}

For the construction of the bases of the rapid 
decay homology groups, we first prove 
some preliminary results. 

\begin{definition}
We define a subset $\Omega_0$ of $\Omega \subset 
X= \CC_z^N$ by: $z \in \Omega_0$ 
$\Longleftrightarrow$ 
$z \in \Omega$ and the Laurent polynomial 
$h_z(x)= \sum_{j=1}^N z_jx^{a(j)}$ has only 
non-degenerate (Morse) critical points in 
$T=( \CC^*)^n$. 
\end{definition}
It is clear that $\Omega_0 \subset X= \CC_z^N$ 
is stable by the multiplication of $\CC^*$ 
(i.e. homothety) on $X= \CC_z^N$. Let $b(1), b(2), 
\ldots, b(n) \in A$ be elements of $A$ such that 
$\{ b(1), b(2), \ldots, b(n) \}$ is a basis of 
the vector space $\RR^n$. By our assumption 
that $A$ generates $\ZZ^n$, we can take such 
elements of $A$. 

\begin{proposition}
Let $h(x)= \sum_{j=1}^N z_jx^{a(j)}$ be a Laurent 
polynomial with support in 
$A \subset \ZZ^n$ on $T=( \CC^*)^n$. 
Assume that $h$ is non-degenerate i.e. 
$z=(z_1, z_2, \ldots, z_N) 
\in \Omega$. Then for generic $\alpha =( \alpha_1, \alpha_2, 
\ldots, \alpha_n) \in \CC^n$ the perturbation 
\begin{equation}
\tilde{h}(x)=h(x)- \sum_{i=1}^n \alpha_i x^{b(i)} 
\end{equation}
of $h$ is non-degenerate and has only 
non-degenerate (Morse) critical points in $T$. 
\end{proposition}

\begin{proof} 
It is clear that $\tilde{h}$ is non-degenerate for 
generic $\alpha \in \CC^n$ (see for example 
\cite[Lemma 5.2]{M-T-4}). 
Let $l_1, l_2, \ldots, 
l_n \in (\RR^n)^*$ be the dual basis of 
$b(1), b(2), \ldots, b(n)$ and set 
\begin{equation}
g_i(x)= \sum_{j=1}^N l_i(a(j)) z_j x^{a(j)-b(i)} 
\qquad (i=1,2, \ldots, n). 
\end{equation}
Note that for $a=(a_1, a_2, \ldots, a_n) \in \RR^n$ 
we have 
\begin{equation}
a_j= \sum_{i=1}^n l_i(a) b(i)_j 
\qquad (j=1,2, \ldots, n). 
\end{equation}
Then we can easily prove the equality 
\begin{equation}
(x_1 \frac{\partial \tilde{h}}{\partial x_1}, 
\ldots, 
x_n \frac{\partial \tilde{h}}{\partial x_n})= 
(x^{b(1)}(g_1- \alpha_1), \ldots, 
x^{b(n)}(g_n- \alpha_n)) \cdot B, 
\end{equation}
where $B \in GL_n( \CC )$ is an invertible matrix 
defined by $B=(b_{ij})_{i,j=1}^n 
=(b(i)_j)_{i,j=1}^n$. Hence we obtain 
\begin{equation}
\{ x \in T \  |  \ 
\frac{\partial \tilde{h}}{\partial x_1}(x)
= \cdots = 
\frac{\partial \tilde{h}}{\partial x_n}(x)=0 \} 
= \{ x \in T \  |  \ 
g_i(x)= \alpha_i \ (1 \leq i \leq n) \}. 
\end{equation}
Moreover degenerate critical points of $\tilde{h}$ 
in $T$ correspond to critical points $x \in T$ of 
the map $(g_1,g_2, \ldots, g_n):T \longrightarrow \CC^n$ 
such that $g_i(x)= \alpha_i$ ($1 \leq i \leq n$). 
By the Bertini-Sard theorem, generic 
$\alpha =( \alpha_1, \alpha_2, 
\ldots, \alpha_n) \in \CC^n$ are not such critical 
values. This implies that for generic 
$\alpha \in \CC^n$ the Laurent polynomial $\tilde{h}$ 
has no degenerate critical point. 
This completes the proof. 
\end{proof} 

\begin{corollary}\label{OPD} 
The subset $\Omega_0$ of $\Omega$ is open 
dense in $X= \CC_z^N$ and stable by 
the multiplication of $\CC^*$ 
(i.e. homothety) on $X= \CC_z^N$. 
\end{corollary}

\begin{proposition}\label{NCP} 
Assume that $0 \in \Int (\Delta )$. Then for 
any $z \in \Omega_0$ the Laurent polynomial 
$h_z(x)= \sum_{j=1}^N z_j x^{a(j)}$ has exactly 
$\Vol_{\ZZ}( \Delta )$ 
non-degenerate (Morse) critical points in $T$. 
\end{proposition}

\begin{proof} 
Let us fix $z \in \Omega_0$ and set $h(x)=h_z(x)$. 
By an invertible matrix $C \in GL_n( \CC )$ we 
define new Laurent polynomials $h_1, h_2, \ldots, h_n$ 
on $T$ by 
\begin{equation}
(h_1, \ldots, h_n)= 
(x_1 \frac{\partial h}{\partial x_1}, 
\ldots, 
x_n \frac{\partial h}{\partial x_n}) \cdot C. 
\end{equation}
By our assumption $0 \in \Int (\Delta )$, 
taking sufficiently generic $C$ 
we may assume that all the Newton polytopes 
of $h_1, h_2, \ldots, h_n$ are equal to $\Delta$. 
Then for any face $\Gamma \prec \Delta$ of $\Delta$ the set 
\begin{equation}
\{ x \in T \  |  \ 
h_1^{\Gamma}(x)= \cdots =h_n^{\Gamma}(x)=0 \}
\end{equation}
coincides with that of the critical points of 
$h^{\Gamma}$ in $T$. In this correspondence 
for the special 
case $\Gamma = \Delta$, multiple roots of 
the equation $h_1(x)= \cdots =h_n(x)=0$ in $T$ 
correspond to degenerate critical points of 
$h:T \longrightarrow \CC$. But by our assumption 
$z \in \Omega_0$ there is no such point in $T$. 
Moreover by the non-degeneracy of $h$ 
($\Longleftrightarrow z \in \Omega$), for any 
face $\Gamma \prec \Delta$ of $\Delta$ such that 
$0 \notin \Gamma$ (i.e. $\Gamma \not= \Delta$ 
when $0 \in \Int ( \Delta )$) we have 
\begin{equation}
\{ x \in T \  |  \ 
h_1^{\Gamma}(x)= \cdots =h_n^{\Gamma}(x)=0 \}
= \emptyset. 
\end{equation}
This means that the ($0$-dimensional) 
subvariety $\{ x \in T \  |  \ 
h_1(x)= \cdots =h_n(x)=0 \}$ of 
$T$ is a non-degenerate complete intersection 
(for the definition, see 
\cite[Definition 2.7]{M-T-2} and 
\cite{Oka}). 
Then by Bernstein's theorem 
its cardinality is 
equal to $\Vol_{\ZZ}( \Delta )$. 
\end{proof}

\subsection{A basis of the rapid decay 
homology group}\label{SS-5}

From now on, assuming the condition 
$0 \in \Int (\Delta )$, for any $z \in \Omega_0$ 
we construct a natural basis of 
the rapid decay homology group $(\CH_{n}^{\rd})_z \simeq 
H_n^{\rd}(U_z; \CK_z^* )$ by using the (relative) 
twisted Morse theory for the function 
${\rm Re}(h_z):T^{\an} \longrightarrow \RR$. 
For the twisted Morse theory and 
its applications to period integrals, we refer to 
Aomoto-Kita \cite{A-K}, Pajitnov \cite{Paj} and 
Pham \cite{Pham}. Our construction of 
the basis is similar to the ones of 
Dubrovin \cite{Dubrovin} and 
Tanabe-Ueda \cite{T-U} 
in the untwisted case. Note that by our 
assumption $0 \in \Int (\Delta )$ any 
parameter vector $c \in \CC^n$ is nonresonant. 
This implies that Theorem \ref{Main} holds 
for any $c \in \CC^n$. 
For $z \in \Omega_0$ let $\alpha (i) \in T$ 
($1 \leq i \leq \Vol_{\ZZ}( \Delta )$) be 
the non-degenerate (Morse) critical points 
of the Laurent polynomial 
$h_z(x)= \sum_{j=1}^N z_jx^{a(j)}$ in 
Proposition \ref{NCP}. By the Cauchy-Riemann 
equation, they are also non-degenerate (Morse) 
critical points of the real-valued function 
${\rm Re}(h_z):T^{\an} \longrightarrow \RR$. 
We can observe this fact more explicitly 
by taking a holomorphic Morse coordinate around 
each $\alpha (i) \in T$ as follows. For a 
fixed $1 \leq i \leq \Vol_{\ZZ}( \Delta )$ 
let $y=(y_1, \ldots, y_n)$, 
$y_j= \xi_j + \sqrt{-1} \eta_j$ ($1 \leq j \leq  
n$) be a holomorphic Morse coordinate for 
$h_z$ around its critical point $\alpha (i) \in T$ 
such that $h_z(x)=h_z (\alpha (i))+ 
y_1^2+y_2^2+ \cdots +y_n^2$ 
in a neighborhood of $\alpha (i) \in T$. 
Since we have 
\begin{equation}
{\rm Re}(h_z)(x)={\rm Re}(h_z) ( \alpha (i))+
( \xi_1^2+ \cdots + \xi_n^2) 
-( \eta_1^2+ \cdots + \eta_n^2), 
\end{equation}
we regard the smooth submanifold $\{ \xi_1= \cdots 
= \xi_n=0 \}$ in it as the stable manifold of 
the gradient flow of the Morse function 
${\rm Re}(h_z):T^{\an} \longrightarrow \RR$ 
in a neighborhood of $\alpha (i) \in T^{\an}$ 
and denote it by $S_i$. By shrinking $S_i$ 
if necessary, we may assume that $\overline{S_i}$ 
is homeomorphic to the $n$-dimensional disk. 
For $1 \leq i \leq \Vol_{\ZZ}( \Delta )$ 
let $R_i \subset \CC$ be the ray in $\CC$ 
defined by 
\begin{equation} 
R_i = \{ \lambda \in \CC \ | \ 
{\rm Re} \lambda \leq {\rm Re} (h_z)( \alpha (i)), 
 \ {\rm Im} \lambda = {\rm Im} 
(h_z)( \alpha (i)) \}. 
\end{equation}
Namely $R_i$ emanates from the critical value 
$h_z( \alpha (i)) \in \CC$ of $h_z$ and 
goes to the left in the complex plane 
$\CC$ so that we have ${\rm Re} 
\lambda \longrightarrow - \infty$ along it. 
By shrinking the stable manifold $S_i$ 
if necessary, we may assume also that the image 
of $\overline{S_i} \subset T^{\an}$ by 
the map $h_z: T^{\an} \longrightarrow \CC$ 
is the closed interval 
\begin{equation} 
R_i^{\e} = \{ \lambda \in R_i \ | \ 
 {\rm Re} (h_z)( \alpha (i)) - \e \leq 
{\rm Re} \lambda \leq {\rm Re} 
(h_z)( \alpha (i)) \} 
\end{equation}
in $R_i$ for some $\e >0$ and 
$h_z( \partial S_i)$ is just the one point 
$\{ h_z( \alpha (i))- \e \}$ in $R_i$. 
We drag $\partial S_i \simeq S^{n-1}$ 
over the complement of $R_i^{\e}$ 
in $R_i$ to construct a tube 
$M_i \simeq (- \infty, 0] \times S^{n-1}$ 
in $T^{\an}$. Finally we set $\gamma_i:=
\overline{S_i} \cup M_i \subset T^{\an}$. 
From now we shall use the notations in 
the proof of Theorem \ref{Main}.  
Then by Proposition \ref{RTH}, for $U_z=T$ 
and the local system 
\begin{equation}
\LL = \CC_{T^{\an}}x_1^{c_1-1} \cdots x_n^{c_n-1}. 
\end{equation}
there exists an isomorphism 
\begin{equation}
H_n^{\rd}(U_z ; \CK_z^* ) \simeq 
H_n(T^{\an} \cup Q, Q; \iota_* \LL ). 
\end{equation}
Since $\gamma_i \subset T^{\an}$ is a singular 
$n$-chain in $T^{\an}$ whose boundary 
in the real oriented blow-up 
$\tilde{Z}$ is contained in 
$Q \subset \tilde{D}$, we obtain 
an element $[ \gamma_i ]$ of the relative 
twisted homology group 
$H_n(T^{\an} \cup Q, Q; \iota_* \LL )$. 
Namely $[ \gamma_i ] \in 
H_n(T^{\an} \cup Q, Q; \iota_* \LL )$ thus 
obtained is a rapid decay $n$-cycle 
in $T^{\an}$ for the function 
\begin{equation}
g_z(x)= \exp (h_z(x)) x_1^{c_1-1} \cdots x_n^{c_n-1}
\end{equation}
satisfying the conditions 
\begin{eqnarray}
(i)  : & S_i \subset \gamma_i, 
\\
(ii)  : & \gamma_i \setminus 
\overline{S_i} \subset 
\{ x \in T^{\an} \ | \ {\rm Re} (h_z) (x) < 
{\rm Re} (h_z)( \alpha (i)) - \varepsilon \} 
\quad \text{for some} \ \varepsilon >0. 
\end{eqnarray}

\begin{theorem}\label{BRH} 
In the situation as above (i.e. $0 \in \Int (\Delta )$ 
and $z \in \Omega_0$), 
the elements $[ \gamma_1], [ \gamma_2], \ldots, 
[ \gamma_{\Vol_{\ZZ}( \Delta )}] \in 
(\CH_{n}^{\rd})_z \simeq H_n^{\rd}
(U_z; \CK_z^* ) \simeq 
H_n(T^{\an} \cup Q, Q; \iota_* \LL )$ 
form a basis of the rapid 
decay homology group 
$H_n^{\rd}(U_z ; \CK_z^* ) \simeq 
H_n(T^{\an} \cup Q, Q; \iota_* \LL )$. 
\end{theorem}

\begin{proof} 
First note that by \eqref{key} we have 
\begin{equation}\label{VOL} 
\dim 
H_n(T^{\an} \cup Q, Q; \iota_* \LL )= 
\Vol_{\ZZ}( \Delta )= \sharp \{ \alpha (i) \}. 
\end{equation}
For $t \in \RR$ we define an open 
subset $T_t^{\an} \subset T^{\an}$ 
of $T^{\an}$ by 
\begin{equation}
T_t^{\an}= \{ x \in T^{\an} \ | \ 
{\rm Re}(h_z) (x) <t \}. 
\end{equation}
Then by Remark \ref{TRAN} 
for any $t \in \RR$ the closure 
of $\partial T_t^{\an} 
\subset T^{\an}$ in $Z$ 
intersects each irrelevant 
divisor $D_i \subset Z$ transversally. 
This implies that for any $p \in \ZZ$ and 
$t \ll 0$ we have 
\begin{equation}\label{EQN1} 
H_p(T_t^{\an} \cup Q, Q; \iota_* \LL ) \simeq 0. 
\end{equation}
Moreover for any $p \in \ZZ$ and 
$t \gg 0$ we have an isomorphism 
\begin{equation}\label{EQN2} 
H_p(T_t^{\an} \cup Q, Q; \iota_* \LL ) \simeq 
H_p(T^{\an} \cup Q, Q; \iota_* \LL ). 
\end{equation}
Now let $- \infty < t_1 < t_2 < \cdots < t_r < 
+ \infty$ be the critical values of 
${\rm Re}(h_z):T^{\an} \longrightarrow \RR$. 
Then by Remark \ref{TRAN} for any $p \in \ZZ$ and 
$s_1, s_2 \in \RR$ such that 
$s_1 < s_2$, $[s_1, s_2] \cap \{ t_1, t_2, 
\ldots, t_r \} = \emptyset$ we have a natural 
isomorphism 
\begin{equation} 
H_p(T_{s_1}^{\an} \cup Q, Q; \iota_* \LL ) \simeq 
H_p(T_{s_2}^{\an} \cup Q, Q; \iota_* \LL ). 
\end{equation}
For $1 \leq j \leq r$ let 
$\alpha (i_1), \alpha (i_2), 
\ldots, \alpha (i_{n_j}) \in T^{\an}$ be the 
critical points of ${\rm Re}(h_z):T^{\an} 
\longrightarrow \RR$ such that 
${\rm Re}(h_z)(\alpha (i_q))=t_j$. 
Then, for sufficiently small 
$0 < \varepsilon \ll 1$ we obtain a short 
exact sequence 
\begin{eqnarray}
0 \longrightarrow  
H_n(T_{t_j - \varepsilon}^{\an} 
\cup Q, Q; \iota_* \LL )
  \longrightarrow 
H_n(T_{t_j - \varepsilon}^{\an} 
\cup ( \cup_{q=1}^{n_j} S_{i_q}) 
\cup Q, Q; \iota_* \LL )
\\
\longrightarrow   
\oplus_{q=1}^{n_j} 
H_n( \overline{S_{i_q}}, \partial S_{i_q} 
 ; \iota_* \LL ) \longrightarrow  0 
\end{eqnarray}
by induction on $j$ 
with the help of \eqref{EQN1} and 
the fact $H_p( \overline{S_{i_q}} , 
\partial S_{i_q}  ; \iota_* \LL ) 
\simeq 0$ $(p \not= n )$. 
Moreover there exist $[S_{i_q}] \in 
H_n( \overline{S_{i_q}}, 
\partial S_{i_q}  ; \iota_* \LL ) 
\simeq \CC$ which can be lifted 
to the elements 
$[ \gamma_{i_q}]$ of $H_n(T_{t_j - \varepsilon}^{\an} 
\cup ( \cup_{q=1}^{n_j} S_{i_q}) 
\cup Q, Q; \iota_* \LL ) \subset 
H_n(T^{\an} \cup Q, Q; \iota_* \LL )$. 
This implies that $[ \gamma_1], [ \gamma_2], \ldots, 
[ \gamma_{\Vol_{\ZZ}( \Delta )}] \in 
H_n(T^{\an} \cup Q, Q; \iota_* \LL )$ 
form a basis of 
$H_n(T^{\an} \cup Q, Q; \iota_* \LL )$. 
\end{proof} 

Note that for a connected 
open neighborhood $V$ of the point 
$z$ in $(\Omega_0)^{\an}$ 
the basis $[ \gamma_1], \ldots, 
[ \gamma_{\Vol_{\ZZ}( \Delta )}] \in 
(\CH_{n}^{\rd})_z$ constructed in Theorem \ref{BRH} 
can be naturally extended to a family 
of the bases $[ \gamma^w_1], \ldots, 
[ \gamma^w_{\Vol_{\ZZ}( \Delta )}] \in (\CH_{n}^{\rd})_w$ 
($w \in V$) i.e. a basis of the local system 
$\CH_{n}^{\rd}$ on $V$.  We can extend it so that 
$V \subset (\Omega_0)^{\an}$ is stable by 
the multiplication of the group 
$\RR_{>0}$ on $X= \CC^N$ and the rapid 
decay $n$-cycles $\gamma^w_1, \ldots, 
\gamma^w_{\Vol_{\ZZ}( \Delta )}$ 
($w \in V$) satisfy the conditions 
\begin{eqnarray}
(i)  : & S^w_i \subset \gamma^w_i, 
\\
(ii)  : & \gamma^w_i \setminus 
\overline{S^w_i} \subset 
\{ x \in T^{\an} \ | \ {\rm Re} (h_w) (x) < 
{\rm Re} (h_w)( \alpha (i)^w ) - \varepsilon \} 
\quad \text{for some} \ \varepsilon >0, 
\end{eqnarray}
where $S^w_i \subset T^{\an}$ is the stable 
manifold  of the gradient flow of 
${\rm Re}(h_w)$ passing through 
its $i$-th non-degenerate 
critical point $\alpha (i)^w \in T^{\an}$. 
For $1 \leq i \leq \Vol_{\ZZ}( \Delta )$ 
we define a confluent $A$-hypergeometric 
function $u_i$ on $V \subset (\Omega_0)^{\an}$ by 
\begin{equation}
u_i(w)= \int_{\gamma_i^w} 
\exp (\sum_{j=1}^N w_j x^{a(j)}) 
x_1^{c_1-1} \cdots x_n^{c_n-1} 
dx_1 \wedge \cdots \wedge dx_n  
\end{equation}
for $w \in V$.

\subsection{Asymptotic expansions 
at infinity}

Now by applying the higher-dimensional 
saddle point (steepest descent) method to holomorphic 
Morse coordinates around the critical points of 
${\rm Re}(h_w):T^{\an} \longrightarrow \RR$ 
($w \in V$) in $T^{\an}$ we obtain 
the following result. For $\delta >0$ let 
$\Lambda_{\delta} \subset \CC$ be the open 
sector in $\CC$ defined by 
$\Lambda_{\delta} = \{ \lambda \in \CC \ | \ 
- \delta < {\rm arg } \lambda < \delta \}$. 
By taking a sufficiently small $\delta >0$ 
such that $\lambda \cdot z \in V$ for any 
$\lambda \in \Lambda_{\delta}$ we set 
$\Lambda :=\Lambda_{\delta}$. 

\begin{theorem}\label{ASE} 
In the situation as above (i.e. 
$0 \in \Int ( \Delta )$), 
if $\delta >0$ is sufficiently small, 
for any $1 \leq i \leq \Vol_{\ZZ}( \Delta )$ 
and $\lambda \in \Lambda$ 
we have an asymptotic 
expansion: 
\begin{eqnarray}
u_i(\lambda \cdot z)=  
\int_{\gamma_i^{\lambda \cdot z}} 
\exp ( \lambda \sum_{j=1}^N z_j x^{a(j)}) 
x_1^{c_1-1} \cdots x_n^{c_n-1} 
dx_1 \wedge \cdots \wedge dx_n  
\\ 
 \ \sim \ ( \sqrt{-1})^n \alpha (i)_1^{c_1-1} \cdots  
\alpha (i)_n^{c_n-1} \times 
\exp (\lambda \cdot h_z( \alpha (i))) 
\\ 
\times    
\left\{ \frac{(2 \pi )^{\frac{n}{2}} }{\sqrt{H_i(z)}} 
\cdot \frac{1}{ \lambda^{\frac{n}{2}} } + 
\frac{\beta_1(z)}{ \lambda^{\frac{n}{2}+1} } + 
\frac{\beta_2(z)}{ \lambda^{\frac{n}{2}+2} } + 
\cdots \cdots \cdots 
\right\} 
\end{eqnarray}
as $| \lambda | \longrightarrow + \infty$ 
in the sector $\Lambda$, where 
$\beta_i(z) \in \CC$ are functions of $z$ and 
\begin{equation}
H_i(z)={\rm det} \left( 
\frac{\partial^2 h_z}{\partial x_j \partial x_k} 
\right)_{x= \alpha (i)}
\end{equation}
is the Hessian of $h_z$ at $x= \alpha (i) \in T^{\an}$. 
\end{theorem}

\begin{proof} First, it is clear that for any 
$\lambda \in \Lambda$ the critical points of 
the function ${\rm Re} (h_{\lambda \cdot z}) 
={\rm Re} ( \lambda \cdot h_{z}) : T^{\an} 
\longrightarrow \RR$ are $\alpha (i)$ ($1 \leq i \leq 
\Vol_{\ZZ}( \Delta )$). 
Fix $1 \leq i \leq \Vol_{\ZZ}( \Delta )$. 
Let $y=(y_1, \ldots, y_n)$, 
$y_j= \xi_j + \sqrt{-1} \eta_j$ ($1 \leq j \leq 
n$) be the holomorphic Morse coordinate for 
the function $h_z$ around 
its $i$-th critical point $\alpha (i) \in T^{\an}$ 
such that $h_z(x)=h_z (\alpha (i))+ 
y_1^2+ \cdots +y_n^2$. 
For sufficiently small $\e >0$ we define an 
open neighborhood $W_{\e}$ of 
$\alpha (i) \in T^{\an}$ by 
$W_{\e}= \{ 
y=(y_1, \ldots, y_n) \ | \ |y_j| < \e \ 
(1 \leq j \leq n)  \} \simeq B(0; \e ) \times 
\cdots \times B(0; \e ) \subset T^{\an}$ and set 
\begin{equation}
S_i^z= \{ \xi_1= \cdots 
= \xi_n=0 \} = \{ \eta \in \RR^n \ | \ 
| \eta_j| < \e \ (1 \leq j \leq n)  \} \subset 
W_{\e}
\end{equation}
in it. Then $S_i^z$ is the stable manifold of 
the gradient flow of the function 
${\rm Re}(h_{z})$ passing through 
its non-degenerate 
critical point $\alpha (i) \in T^{\an}$. 
For $\lambda = | \lambda | e^{\sqrt{-1} \theta} 
\in \Lambda \subset \CC$ 
($- \delta < \theta ={\rm arg} \lambda 
< \delta$) we set 
$(y_1^{\prime}, \ldots, y_n^{\prime})
=(e^{\frac{\sqrt{-1} \theta}{2}}y_1, \ldots, 
e^{\frac{\sqrt{-1} \theta}{2}}y_n)$. 
Then the Laurent 
polynomial $h_{\lambda \cdot z} 
= \lambda \cdot h_z$ can be rewritten as 
\begin{equation}
h_{\lambda \cdot z}(x)= 
\lambda \cdot h_z( \alpha (i))
+ | \lambda | (y_1^{\prime})^2 + \cdots + 
| \lambda | (y_n^{\prime})^2. 
\end{equation}
By setting $y_j^{\prime}
= \xi_j^{\prime} + \sqrt{-1} \eta_j^{\prime}$ 
($1 \leq j \leq n$) we see also that the subset 
\begin{equation}
S_i^{\lambda \cdot z}=\{ \xi_1^{\prime}= 
\cdots = \xi_n^{\prime}=0 \} = 
\{ \eta^{\prime} \in \RR^n \ | \ 
| \eta_j^{\prime} | < \e \ (1 \leq j \leq n)  \} 
\subset W_{\e}
\end{equation}
of $W_{\e}$ is the stable manifold of 
the gradient flow of  
${\rm Re}(h_{\lambda \cdot z})$ through 
$\alpha (i) \in T^{\an}$. 
By our construction of the rapid decay 
$n$-cycle $\gamma^{\lambda \cdot z}_i$ 
we may assume that 
$S^{\lambda \cdot z}_i 
\subset \gamma^{\lambda \cdot z}_i$ and 
\begin{equation}
{\rm Re} (h_{\lambda \cdot z})(x) - 
{\rm Re} (h_{\lambda \cdot z})( \alpha (i)) 
< - \e^2 | \lambda |, \quad 
{\rm Im} (h_{\lambda \cdot z})(x) = 
{\rm Im} (h_{\lambda \cdot z})( \alpha (i)) 
\end{equation}
for any $x \in \gamma^{\lambda \cdot z}_i 
\setminus \overline{S^{\lambda \cdot z}_i}$. 
We may assume also that for any 
$\lambda, \lambda^{\prime} \in 
\Lambda$ such that ${\rm arg} \lambda = 
{\rm arg} \lambda^{\prime}$ we have 
$\gamma^{\lambda \cdot z}_i=
\gamma^{\lambda^{\prime} \cdot z}_i$. 
This implies that 
(if $\delta >0$ is sufficiently small) there exists 
a positive real numbers $C>0$ such that 
\begin{equation}
\int_{ \gamma^{\lambda \cdot z}_i 
\setminus \overline{S^{\lambda \cdot z}_i} }  
 \Bigl| \exp ( h_z(x)-h_z( \alpha (i) ) ) 
x_1^{c_1-1} \cdots x_n^{c_n-1} 
dx_1 \wedge \cdots \wedge dx_n \Bigr| <C
\end{equation}
for any $\lambda \in \Lambda$. Then we have 
\begin{eqnarray}
 \Bigl| \int_{ \gamma^{\lambda \cdot z}_i 
\setminus \overline{S^{\lambda \cdot z}_i} }  
\exp ( \lambda \cdot h_z(x) ) 
x_1^{c_1-1} \cdots x_n^{c_n-1} 
dx_1 \wedge \cdots \wedge dx_n \Bigr| 
 \ = \ | \exp ( \lambda \cdot h_z( \alpha (i)) ) | 
\\ 
\times  
 \Bigl| \int_{ \gamma^{\lambda \cdot z}_i 
\setminus \overline{S^{\lambda \cdot z}_i} }  
\exp (  \lambda \cdot 
\{ h_z(x)-h_z( \alpha (i)) \} ) 
x_1^{c_1-1} \cdots x_n^{c_n-1} 
dx_1 \wedge \cdots \wedge dx_n \Bigr| 
\\
\leq C | \exp ( \lambda \cdot 
h_z( \alpha (i)) ) | \times 
\sup_{ x \in \gamma^{\lambda \cdot z}_i 
\setminus \overline{S^{\lambda \cdot z}_i} } 
 \Bigl| \exp \{ \frac{\lambda -1}{\lambda} 
(h_{\lambda \cdot z}(x)- 
h_{\lambda \cdot z}( \alpha (i))) \} \Bigr| 
\\
\leq C | \exp ( \lambda \cdot 
h_z( \alpha (i)) ) | \times 
\exp (- \frac{\e^2}{2} | \lambda | )
\end{eqnarray}
for any $\lambda \in \Lambda$ satisfying 
$| \lambda | \gg 0$. Hence, to prove the theorem, 
it suffices to calculate the asymptotic expansion of 
the integral 
\begin{equation}
\tl{u_i}( \lambda \cdot z )=  
 \int_{S_i^{\lambda \cdot z}}  
\exp ( \lambda \sum_{j=1}^N z_j x^{a(j)}) 
x_1^{c_1-1} \cdots x_n^{c_n-1} 
dx_1 \wedge \cdots \wedge dx_n  
\end{equation}
as $| \lambda | \longrightarrow + \infty$ 
in the sector $\Lambda$. 
For the  Morse coordinate $y=(y_1, \ldots, y_n)$ of 
$h_z$ we can easily show 
\begin{equation}\label{COE} 
{\rm det} \left( 
\frac{\partial y_j}{\partial x_k} 
\right)_{x= \alpha (i)}=
\sqrt{\frac{H_i(z)}{2^n}}. 
\end{equation}
Also by using the coordinate 
$y=(y_1, \ldots, y_n)$ set 
\begin{equation}
f(y_1, \ldots, y_n):=x_1^{c_1-1} \cdots x_n^{c_n-1} 
\times {\rm det} \left( 
\frac{\partial x_j}{\partial y_k} \right)
\end{equation}
and let 
\begin{equation}
f(y_1, \ldots, y_n)= \sum_{a \in \ZZ_+^n} 
f_a y^a \quad (f_a \in \CC ) 
\end{equation}
be its Taylor expansion at $y=0$ 
i.e. $x= \alpha (i)$. Then by 
\eqref{COE} we obtain 
\begin{equation}\label{HES} 
f_0=f(0, \ldots, 0)= \alpha (i)_1^{c_1-1} \cdots  
\alpha (i)_n^{c_n-1} \times 
\sqrt{ \frac{2^n}{H_i(z)} }. 
\end{equation}
Now the restriction of the $n$-form 
\begin{equation}
\exp ( \lambda \cdot h_z(x)) 
x_1^{c_1-1} \cdots x_n^{c_n-1} 
dx_1 \wedge \cdots \wedge dx_n
\end{equation}
to the stable manifold 
$S_i^{\lambda \cdot z} 
=  \{ \eta^{\prime} \in \RR^n \ | \ 
| \eta_j^{\prime} | < \e \ (1 \leq j \leq n) \} 
\subset \RR^n$ 
has the following form: 
\begin{eqnarray}
( \sqrt{-1})^n e^{- \frac{\sqrt{-1} n \theta}{2}} 
\exp \left\{ 
\lambda \cdot h_z( \alpha (i))-
| \lambda | ( \eta_1^{\prime})^2 - \cdots - 
| \lambda | ( \eta_n^{\prime})^2 
\right\}  
\\ 
\times 
\left[
 \sum_{a \in \ZZ_+^n} f_a \cdot 
e^{- \frac{\sqrt{-1} |a| \theta}{2}} \cdot 
\{ \sqrt{-1} \eta^{\prime} \}^a 
\right] 
d \eta_1^{\prime} \wedge \cdots 
\wedge d \eta_n^{\prime}. 
\end{eqnarray}
For any $a=(a_1, \ldots, a_n) \in \ZZ_+^n$ 
we can show that the integral of the 
$n$-form 
\begin{eqnarray}
\omega_a := 
( \sqrt{-1})^n e^{- \frac{\sqrt{-1} n \theta}{2}} 
\exp \left\{ 
\lambda \cdot h_z( \alpha (i))-
| \lambda | ( \eta_1^{\prime})^2 - \cdots - 
| \lambda | ( \eta_n^{\prime})^2 
\right\}  
\\ 
\times 
f_a \cdot e^{- \frac{\sqrt{-1} 
|a| \theta}{2}} \cdot 
\{ \sqrt{-1} \eta^{\prime} \}^a 
d \eta_1^{\prime} \wedge \cdots 
\wedge d \eta_n^{\prime} 
\end{eqnarray}
over the whole $\RR^n_{\eta^{\prime}}$ is 
equal to 
\begin{eqnarray}
( \sqrt{-1})^n \lambda^{- \frac{n}{2}} 
\exp ( \lambda \cdot h_z( \alpha (i))) \times 
f_a \cdot \lambda^{- \frac{|a|}{2}} 
\\ 
\times 
\int_{\RR^n}
\exp (-t_1^2- \cdots -t_n^2) 
\{ \sqrt{-1}t \}^a 
d t_1 \wedge \cdots \wedge d t_n 
\end{eqnarray}
by setting $(t_1, \ldots, t_n)= 
(\sqrt{| \lambda |} \eta_1^{\prime}, \ldots, 
\sqrt{| \lambda |} \eta_n^{\prime})$. 
Note that the integral 
\begin{equation}
\int_{\RR^n}
\exp (-t_1^2- \cdots -t_n^2) 
\{ \sqrt{-1}t \}^a 
d t_1 \wedge \cdots \wedge d t_n 
\end{equation}
is zero if $a_i \in \ZZ_+$ is odd for some 
$1 \leq i \leq n$. As the previous part 
of this proof, we can show also that 
there exists $M>0$ such that 
\begin{equation}
 \Bigl| \int_{\RR^n \setminus S_i^{\lambda \cdot z}}
 \omega_a \Bigr| \leq M 
 | \exp ( \lambda \cdot 
h_z( \alpha (i)) ) | \times 
\exp (- \frac{\e^2}{2} | \lambda | )
\end{equation}
for any $\lambda \in \Lambda$ satisfying 
$| \lambda | \gg 0$. 
Then the result follows immediately from 
\eqref{HES}. This completes the proof. 
\end{proof} 

\begin{remark}
If $z \in \Omega_0$ and the critical 
point $\alpha (i)$ 
of ${\rm Re}(h_z):T^{\an} \longrightarrow \RR$ 
in $T^{\an}$ is given, by using the holomorphic 
Morse coordinate in the proof above 
we can calculate also the coefficients 
$\beta_1(z), \beta_2(z), \ldots \in \CC$ explicitly. 
\end{remark} 

For the point $z \in \Omega_0$ let 
\begin{equation}
\mathbb{L}_z 
= \{ \lambda \cdot z \in X= \CC^N \ | \ 
\lambda \in \CC \} \simeq \CC_{\lambda} 
\end{equation}
be the complex line in $X= \CC^N$ passing through 
$z \in \Omega_0$. Then by 
Theorems \ref{BRH} and \ref{ASE} we can 
observe Stokes' phenomena for the restrictions 
$u_i|_{\mathbb{L}_z}$ 
of the functions $u_i$ ($1 \leq i \leq 
\Vol_{\ZZ}( \Delta )$) to the line $\mathbb{L}_z 
\simeq \CC_{\lambda}$. Indeed, by Theorem \ref{ASE} 
the dominance ordering of the 
functions $u_i|_{\mathbb{L}_z}$ 
at infinity (i.e. where $| \lambda | \gg 0$) 
changes as ${\rm arg}( \lambda )$ increases. 
More precisely, the asymptotic expansions at infinity 
of the restrictions of the $A$-hypergeometric 
functions to $\mathbb{L}_z 
\simeq \CC_{\lambda}$ may jump at the Stokes lines: 
\begin{equation}
\left\{ \lambda \in \CC \ | \ 
{\rm Re} \left[ \lambda \cdot \{ h_z( \alpha (i)) 
- h_z( \alpha (j)) \} \right] =0 
\right\}  \qquad (i \not= j). 
\end{equation}
It would be an interesting problem to 
determine the Stokes multipliers in 
this case. 

\begin{remark}
When $\Delta^{\prime} = 
{\rm conv}(A)$ does not contain the origin 
$0 \in \RR^n$, the last half of the proof of 
Proposition \ref{NCP} does not work. Namely 
for $z \in \Omega_0$ 
the number of the non-degenerate (Morse) 
critical points of 
$h_z(x)$ may be smaller than 
$\Vol_{\ZZ}( \Delta^{\prime} ) < 
\Vol_{\ZZ}( \Delta )$. 
Nevertheless, as in Theorems \ref{BRH} 
and \ref{ASE} we can construct a part of 
a basis of $H_n^{\rd}(U_z ; \CK_z^* ) \simeq 
H_n(T^{\an} \cup Q, Q; \iota_* \LL )$ 
by the corresponding rapid decay $n$-cycles 
and obtain the asymptotic 
expansions at infinity of the 
confluent $A$-hypergeometric functions 
associated to them. Note that 
the number of the critical points of $h_z(x)$ 
in such a case is given by 
\cite[Lemma 2.10]{E-2}. 
\end{remark}

\section{The two-dimensional case}\label{sec:6}

In this section, we shall construct a natural 
basis of the rapid decay homology group 
$H_n (T^{\an} \cup Q, Q; \iota_* \LL )$ in 
the two-dimensional case i.e. $n=2$.

\subsection{Some results 
on relative twisted homology groups}

First, we prepare some elementary results 
on relative twisted homology groups. 
Set $Z= \CC^2_{x_1, x_2}$ and let $h_0$ 
be the meromorphic function on $Z^{\an}$ defined by 
\begin{equation}
h_0(x_1,x_2)=\frac{1}{x_1^{m_1} x_2^{m_2}} \qquad 
(m_1, m_2 \in \ZZ_{>0}). 
\end{equation}
Let $\pi_0 : \tl{Z_0} 
\longrightarrow Z^{\an}$ be the real oriented 
blow-up of $Z^{\an}$ along the 
normal crossing divisor $D_0^{\an}= 
\{ x_1=0 \} \cup \{ x_2=0 \}$ and set $\tl{D_0} = 
\pi_0^{-1}(D_0^{\an}) \subset \tilde{Z_0}$ and 
$U_0^{\an}=Z^{\an} \setminus D_0^{\an} 
\simeq (\CC^*)^2$. 
By the inclusion map 
$\iota_0 : U_0^{\an} \hookrightarrow \tl{Z_0}$ 
we consider $U_0^{\an}$ as an open 
subset of $\tl{Z_0}$ and set 
\begin{equation}
P_0= \tl{D_0} \cap 
\overline{ \{ x \in U_0^{\an}  \ | \ 
{\rm Re} h_0(x) \geq 0 \} } 
\end{equation}
and $Q_0=\tl{D_0} \setminus P_0$. 
Finally let $\LL_0$ be the local system 
of rank one on $U_0^{\an}$ defined by 
\begin{equation}
\LL_0=\CC_{U_0^{\an}}x_1^{\beta_1} x_2^{\beta_2} 
\qquad ( \beta =( \beta_1, \beta_2) \in \CC^2). 
\end{equation}
Then by homotopy and Lemma \ref{ECL} 
we obtain the following lemma. 

\begin{lemma}\label{2-dim-1} 
\par \noindent (i) For $0< \e \ll 1$ set 
\begin{equation}
U_0^{\an}(\e )=\{ x=(x_1,x_2) \in U_0^{\an} 
\ | \ \e < |x_1| < \frac{1}{\e} \} \subset 
U_0^{\an}. 
\end{equation}
Then for any $p \in \ZZ$ the natural morphism 
\begin{equation}
H_p(U_0^{\an}( \e ) \cup Q_0, 
Q_0; (\iota_0)_* \LL_0  ) \longrightarrow 
H_p(U_0^{\an} \cup Q_0, Q_0; 
(\iota_0)_* \LL_0  )
\end{equation}
is an isomorphism. 
\par \noindent (ii) Assume that 
$\beta =( \beta_1, \beta_2) \in \CC^2$ 
satisfies the condition $m_2 \beta_1- 
m_1 \beta_2 \notin \ZZ$. Then we have 
\begin{equation}
H_p(U_0^{\an} \cup Q_0, Q_0; 
(\iota_0)_* \LL_0 ) \simeq 0
\end{equation}
for any $p \in \ZZ$. 
\end{lemma}

\begin{proof} The assertion (i) can be easily shown 
by homotopy. We will prove (ii). Let $S^1$ be 
the unit circle $\{ x_1 \in \CC \ | \ |x_1|=1 \}$ 
in $\CC^1_{x_1}$. Then by (i) and homotopy we 
have an isomorphism 
\begin{equation}
H_p( (S^1 \times \CC^*) \cup Q_0, 
Q_0; (\iota_0)_* \LL_0  ) \simto 
H_p(U_0^{\an} \cup Q_0, Q_0; 
(\iota_0)_* \LL_0  )
\end{equation}
for any $p \in \ZZ$. Let us take the base point 
$e:=1 \in S^1$ of $S^1$. Then by Lemma \ref{ECL} 
we have 
\begin{equation}
H_p( ( \{ e \} \times \CC^*) \cup Q_0, 
Q_0; (\iota_0)_* \LL_0  ) \simeq 0
\end{equation}
for $p \not= 1$ and there exists a natural 
basis $[ \gamma_1], \ldots, [ \gamma_{m_2}]$ 
of $H_1( ( \{ e \} \times \CC^*) \cup Q_0, 
Q_0; (\iota_0)_* \LL_0  ) \simeq \CC^{m_2}$. Let 
\begin{equation}
\Psi_0: H_1( ( \{ e \} \times \CC^*) \cup Q_0, 
Q_0; (\iota_0)_* \LL_0  ) \simto 
H_1( ( \{ e \} \times \CC^*) \cup Q_0, 
Q_0; (\iota_0)_* \LL_0  )
\end{equation}
be the linear automorphism, i.e. the monodromy 
of $H_1( ( \{ e \} \times \CC^*) \cup Q_0, 
Q_0; (\iota_0)_* \LL_0  )$ induced by the 
(clockwise) rotation along the circle $S^1$. 
By the matrix representation of $\Psi_0$ 
with respect to the basis $[ \gamma_1], 
\ldots, [ \gamma_{m_2}]$ we see that the 
eigenvalues of $\Psi_0$ are 
contained in the set 
\begin{equation}
\left\{ t \in \CC \ | \ t^{m_2}= 
\exp [ 2 \pi \sqrt{-1} (m_2 \beta_1- 
m_1 \beta_2) ] \right\}.  
\end{equation}
In particular, our assumption $m_2 \beta_1- 
m_1 \beta_2 \notin \ZZ$ implies that 
$\id - \Psi_0$ is an automorphism of 
$H_1( ( \{ e \} \times \CC^*) \cup Q_0, 
Q_0; (\iota_0)_* \LL_0  )$. Now for 
$0 < \e <<1$ we define two arcs 
$G_{\pm} \subset S^1$ in $S^1$ by 
\begin{equation}
G_{\pm} = \{ x_1 \in S^1 \ | \ 
 \pm {\rm Re} x_1 > - \e 
 | {\rm Im} x_1| \} \subset S^1.  
\end{equation}
Then $S^1=G_+ \cup G_-$. By the Mayer-Vietoris 
exact sequence for relative twisted homology 
groups associated to the open covering 
$S^1 \times \CC^* 
=(G_+ \times \CC^*) \cup (G_- \times \CC^*)$ 
of $S^1 \times \CC^*$ we can calculate 
$H_p( (S^1 \times \CC^*) \cup Q_0, 
Q_0; (\iota_0)_* \LL_0  )$ ($p \in \ZZ$) from 
$H_p( (G_{\pm} \times \CC^*) \cup Q_0, 
Q_0; (\iota_0)_* \LL_0  ) \simeq 
H_p( ( \{ e \} \times \CC^*) \cup Q_0, 
Q_0; (\iota_0)_* \LL_0  )$ ($p \in \ZZ$). 
Then the assertion (ii) follows 
from the invertibility of $\id - \Psi_0$. 
\end{proof} 

Next consider the meromorphic function 
$h_1$ on $Z^{\an}=\CC^2$ defined by 
\begin{equation}
h_1(x_1,x_2)=\frac{1}{(x_1- \lambda_1)^{n_1}
 \cdots (x_1- \lambda_k)^{n_k} 
x_1^{m_1} x_2^{m_2}} \qquad 
(n_j, m_1, m_2 \in \ZZ_{>0}), 
\end{equation}
where $\lambda_1, \ldots, \lambda_k$ are 
distinct non-zero complex numbers. 
Let $\pi_1 : \tl{Z_1} 
\longrightarrow Z^{\an}$ be the real oriented 
blow-up of $Z^{\an}$ along $D_1^{\an}= 
\cup_{j=1}^k \{ x_1= \lambda_j \} \cup 
\{ x_1=0 \} \cup \{ x_2=0 \}$ and define 
$\tl{D_1} \subset \tl{Z_1}$, 
$\iota_1 : U_1^{\an} 
=Z^{\an} \setminus D_1^{\an} 
\hookrightarrow \tl{Z_1}$, 
$P_1 \subset \tl{D_1}$ and 
$Q_1=\tl{D_1} \setminus P_1$ as above. 
Moreover let $\LL_1$ be 
the local system of rank one on $U_1^{\an}$ 
defined by 
\begin{equation}
\LL_1=\CC_{U_1^{\an}}x_1^{\beta_1} x_2^{\beta_2} 
\prod_{j=1}^k (x_1- \lambda_j)^{\beta_j^{\prime}} 
\qquad ( \beta =( \beta_1, \beta_2) \in \CC^2, 
\ \beta^{\prime}=
(\beta_1^{\prime}, \ldots, \beta_k^{\prime}) 
\in \CC^k ). 
\end{equation}
Then by the proof of Lemma \ref{2-dim-1} (ii) 
and Mayer-Vietoris exact sequences for 
relative twisted homology groups we 
obtain the following proposition. 

\begin{proposition}\label{2-dim-2} 
\par \noindent (i) For $0< \e \ll $ set 
\begin{equation}
U_1^{\an}(\e )=\{ (x_1,x_2) \in U_1^{\an} 
\ | \ \e < |x_1| < \frac{1}{\e}, \quad 
 |x_1- \lambda_j| > \e \quad (1 \leq j \leq k) 
\} \subset 
U_1^{\an}. 
\end{equation}
Then for any $p \in \ZZ$ the natural morphism 
\begin{equation}
H_p(U_1^{\an}( \e ) \cup Q_1, 
Q_1; (\iota_1)_* \LL_1 ) \longrightarrow 
H_p(U_1^{\an} \cup Q_1, Q_1; 
(\iota_1)_* \LL_1 )
\end{equation}
is an isomorphism. 
\par \noindent (ii) Assume that $k \geq 1$ and 
$\beta =( \beta_1, \beta_2) \in \CC^2$, 
$\beta^{\prime}=
(\beta_1^{\prime}, \ldots, \beta_k^{\prime}) 
\in \CC^k$ 
satisfy the conditions $m_2 \beta_1- 
m_1 \beta_2 \notin \ZZ$ and 
$m_2 \beta_j^{\prime} - n_j \beta_2 
\notin \ZZ$ for any $1 \leq j \leq k$. Then 
we have 
\begin{equation}
\dim H_p (U_1^{\an} \cup Q_1, Q_1; 
(\iota_1)_* \LL_1 )
=\begin{cases}
k \times m_2 & (p=2), \\
\ 0 & (p \not= 2) 
\end{cases}
\end{equation}
and can explicitly construct a basis 
of the vector space $H_2(U_1^{\an} \cup Q_1, Q_1; 
(\iota_1)_* \LL_1 )$ over $\CC$. 
\end{proposition}

In the special but important case where $k \geq 2$ 
and $n_1=n_2= \cdots =n_k>0$, we can 
construct the basis of $H_2(U_1^{\an} \cup Q_1, Q_1; 
(\iota_1)_* \LL_1 )$ in Proposition \ref{2-dim-2} (ii) 
very elegantly as follows. By homotopy we may 
assume that $\lambda_j = 
\exp ( \frac{2 \pi j}{k} \sqrt{-1})$ 
($1 \leq j \leq k$) from the start. For 
$1 \leq j \leq k$ let $G_j \subset S^1= 
\{ x_1 \in \CC \ | \ |x_1|=1 \}$ be the arc 
in the unit circle $S^1$ between the two points 
$\lambda_j, \lambda_{j+1} \in S^1$, 
where we set $\lambda_{k+1}= \lambda_1$. For 
sufficiently small $\e >0$ let $F_j$ be the 
boundary of the set 
\begin{equation}
B( \lambda_j; \e ) \cup G_j \cup 
B( \lambda_{j+1} ; \e ) \subset \CC^1_{x_1} 
\end{equation}
and denote the central point 
of the arc $G_j$ by $e_j \in G_j$.

\begin{center}
\includegraphics[height=4.5cm]{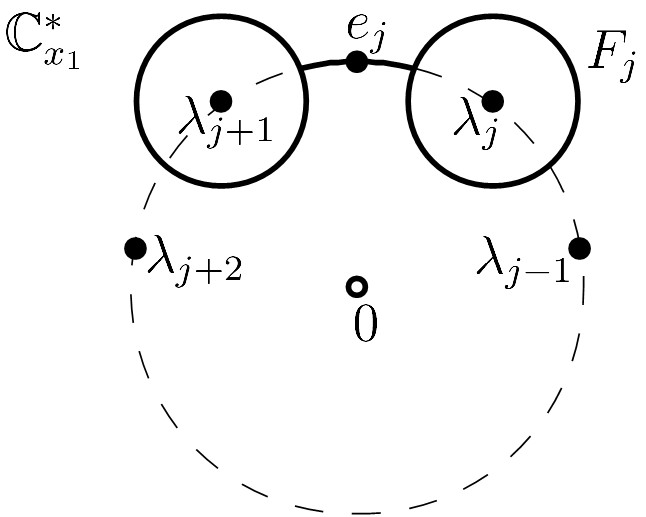}

Figure 4.
\end{center}


We regard $e_j \in F_j$ as the base point of 
the one-dimensional complex $F_j$. 
By Lemma \ref{ECL} we have 
\begin{equation}
H_p( ( \{ e_j \} \times \CC^*) \cup Q_1, 
Q_1; (\iota_1)_* \LL_1  ) \simeq 0
\end{equation}
for $p \not= 1$ and there exists a natural 
basis $[ \gamma_{j1}], \ldots, [ \gamma_{j m_2}]$ 
of $H_1( ( \{ e_j \} \times \CC^*) \cup Q_1, 
Q_1; (\iota_1)_* \LL_1  ) \simeq \CC^{m_2}$. 
Moreover by the proof of Lemma \ref{2-dim-1} (ii) 
and Mayer-Vietoris exact sequences we obtain 
\begin{equation}
H_p( (F_j \times \CC^*) \cup Q_1, 
Q_1; (\iota_1)_* \LL_1  ) \simeq 0
\end{equation}
for $p \not= 2$. Note that the shape of 
$F_j$ looks like that of the figure-$8$. 
We start from the base point  $e_j \in F_j$, 
go along the one-dimensional complex $F_j$ 
in the way of the usual drawing of the figure-$8$
and come back to the same place $e_j \in F_j$. 
Along this path on $F_j$ 
we drag the twisted $1$-cycles 
$\gamma_{j1}, \ldots, \gamma_{j m_2}$ 
over the point $e_j \in F_j$  
keeping their end points in 
the rapid decay direction $Q_1$ of 
$\exp (h_1)$. Then by our 
assumption $n_j=n_{j+1}$ we obtain 
the twisted $2$-cycles 
$[ \delta_{j1}], \ldots, [ \delta_{j m_2}]$ in 
$H_2( (F_j \times \CC^*) \cup Q_1, 
Q_1; (\iota_1)_* \LL_1  )$. 
It is easy to see that they form a basis of 
$H_2( (F_j \times \CC^*) \cup Q_1, 
Q_1; (\iota_1)_* \LL_1  ) \simeq \CC^{m_2}$. 
On the other hand, by Proposition \ref{2-dim-2} 
(i) and homotopy there exists an isomorphism 
\begin{equation}
H_p(  \{ ( \cup_{j=1}^k F_j) \times \CC^* \} \cup Q_1, 
Q_1; (\iota_1)_* \LL_1  ) \simto 
H_p( U_1^{\an} \cup Q_1, Q_1; (\iota_1)_* \LL_1  ) 
\end{equation}
for any $p \in \ZZ$. Moreover it follows from 
our assumption 
$m_2 \beta_j^{\prime} - n_j \beta_2 \notin \ZZ$ 
that we have 
\begin{equation}
H_p( \{ (F_j \cap F_{j-1}) \times \CC^* \} 
\cup Q_1, Q_1; (\iota_1)_* \LL_1  ) \simeq 0
\end{equation}
for any $1 \leq j \leq k$ and $p \in \ZZ$. 
Hence by the Mayer-Vietoris exact sequences 
associated to the covering 
$( \cup_{j=1}^k F_j) \times \CC^* = 
\cup_{j=1}^k (F_j \times \CC^*)$ of 
$( \cup_{j=1}^k F_j) \times \CC^*$ 
we obtain the following result. 

\begin{proposition}\label{2-dim-new} 
Assume that $k \geq 2$, 
$n_1=n_2= \cdots =n_k>0$ and 
$\beta =( \beta_1, \beta_2) \in \CC^2$, 
$\beta^{\prime}=
(\beta_1^{\prime}, \ldots, \beta_k^{\prime}) 
\in \CC^k$ satisfy the condition 
$m_2 \beta_j^{\prime} - n_1 \beta_2 
\notin \ZZ$ for $1 \leq j \leq k$. Then the elements 
$[ \delta_{j1}], \ldots, [ \delta_{j m_2}] 
\in H_2(U_1^{\an} \cup Q_1, Q_1; 
(\iota_1)_* \LL_1 )$ ($1 \leq j \leq k$) 
constructed above are linearly 
independent over $\CC$ and 
form a basis of $H_2(U_1^{\an} \cup Q_1, Q_1; 
(\iota_1)_* \LL_1 )$. 
\end{proposition}

\subsection{A construction of the basis in 
the two-dimensional case}

Now let us consider the situation 
in Sections \ref{sec:4} and \ref{sec:5} in the 
two-dimensional case. For $z \in \Omega$ we 
define $Q \subset \tilde{D} \subset \tilde{Z}$ 
in the real oriented blow-up $\pi : \tilde{Z} 
\longrightarrow Z^{\an}$ of $Z^{\an}=
( \tl{Z_{\Sigma}} )^{\an}$ as in 
the proof of Theorem \ref{Main}. 
For the local system 
$\LL =\CC_{T^{\an}} x_1^{c_1-1} 
x_2^{c_2-1}$ on $T^{\an}$ we shall construct 
a basis of the rapid decay homology 
group $H_2^{\rd}(T^{\an}):=H_2 
(T^{\an} \cup Q, Q; \iota_* \LL )$. By 
abuse of notations, for an open subset 
$W$ of $T^{\an}$ and $p \in \ZZ$ we set 
\begin{equation}
H_p^{\rd}(W):=H_p(W \cup Q, Q; \iota_* \LL )
\end{equation}
for short. Recall that $\Sigma$ is a smooth 
subdivision of the dual fan of $\Delta = 
{\rm conv} ( A \cup \{ 0 \} ) \subset \RR^2$ 
and $\rho_1, \ldots, \rho_{l} \in \Sigma$ are 
the rays i.e. the one-dimensional cones in 
$\Sigma$ which correspond to the relevant 
divisors $D_1, \ldots, D_{l}$ 
in $Z= \tl{Z_{\Sigma}}$. We renumber 
$\rho_1, \ldots, \rho_{l}$ 
in the clockwise order so that we have 
$D_i \cap D_{i+1} \not= \emptyset$ for 
any $1 \leq i \leq l-1$. 
By the primitive vector $\kappa_i \in \rho_i 
\cap (\ZZ^2 \setminus \{ 0 \} )$ on 
$\rho_i$ the order $m_i>0$ of the pole 
of $h_z(x)=\sum_{j=1}^N z_j x^{a(j)}$ 
along $D_i$ is explicitly given by 
\begin{equation}
m_i= - \min_{a \in \Delta} 
\langle \kappa_i, a \rangle . 
\end{equation}
For $1 \leq i \leq l$ we set 
\begin{equation}
\beta_i= \langle \kappa_i, (c_1-1, c_2-1) 
\rangle \in \CC. 
\end{equation}
Then at each point of 
$D_i \setminus (\cup_{j \not= i}D_j)$ 
there exists a local coordinate system 
$(y_1,y_2)$ of $Z_{\Sigma}$ such that 
$D_i= \{ y_1=0 \}$ and the function 
$x_1^{c_1-1} x_2^{c_2-1}$ has the 
form $y_1^{\beta_i}$. Namely 
the function $x_1^{c_1-1} x_2^{c_2-1}$ 
has the order $\beta_i \in \CC$ along $D_i$. 
By the non-degeneracy 
of $h_z$ the complex curve $\overline{h_z^{-1}(0)} 
\subset Z_{\Sigma}$ intersects each relevant 
divisor $D_i$ transversally. Set 
$v_i = \sharp 
(D_i \cap \overline{h_z^{-1}(0)})
\geq 0$ and $\{ q_{i1}, \ldots, q_{iv_i} \} 
=D_i \cap \overline{h_z^{-1}(0)}$. By our 
construction of the complex blow-up 
$Z=\tl{Z_{\Sigma}} \longrightarrow 
Z_{\Sigma}$ of $Z_{\Sigma}$, the fiber 
of the point $q_{ij}$ is a union 
$E_{ij}=E_{ij1} \cup \cdots \cup E_{ijm_i}$ 
of exceptional divisors $E_{ijk} \simeq 
{\mathbb{P}}^1$ ($1 \leq k \leq m_i$). 
Here $E_{ijk}$ is 
the exceptional divisor constructed by the 
$k$-th blow-up over $q_{ij}$. 
See Figure $5$ below. Then the order of the 
pole of (the meromorphic extension of) 
$h_z$ along $E_{ijk}$ is $m_i-k$. 
Moreover the order of 
the function $x_1^{c_1-1} x_2^{c_2-1}$ 
along $E_{ijk}$ is $\beta_i$ for any 
$1 \leq k \leq m_i$. 
Let $T_i \simeq \CC^* \subset D_i$ be 
the one-dimensional $T$-orbit in 
$Z_{\Sigma}$ which corresponds to $\rho_i$ 
and denote by the same letter $T_i$ its 
strict transform in the blow-up 
$Z=\tl{Z_{\Sigma}}$. 
Assume that $\beta_i \notin \ZZ$ ($1 \leq i \leq l$) 
and $m_{i+1} \beta_i- m_i \beta_{i+1} 
\notin \ZZ$ ($1 \leq i \leq l-1$). 
Then by Propositions 
\ref{2-dim-2} (ii) and \ref{2-dim-new} 
there exists a 
sufficiently small tubular neighborhood $W_i$ 
of $T_i^{\an}$ in $Z^{\an}$ such that for 
its open subset $W_i^{\circ}=W_i \cap 
T^{\an} \subset T^{\an}$ we have 
\begin{equation}
\dim H_p^{\rd} (W_i^{\circ})
=\begin{cases}
v_i \times m_i & (p=2), \\
\ 0 & (p \not= 2). 
\end{cases}
\end{equation}
Furthermore we can explicitly construct a basis 
$\delta_{ijk} \in H_2^{\rd} (W_i^{\circ})$ 
($1 \leq j \leq v_i$, $1 \leq k \leq m_i$) 
of the vector space 
$H_2^{\rd} (W_i^{\circ})
=H_2(W_i^{\circ} \cup Q, Q; \iota_* \LL )$ 
over $\CC$.

\begin{center}
\includegraphics[width=12cm]{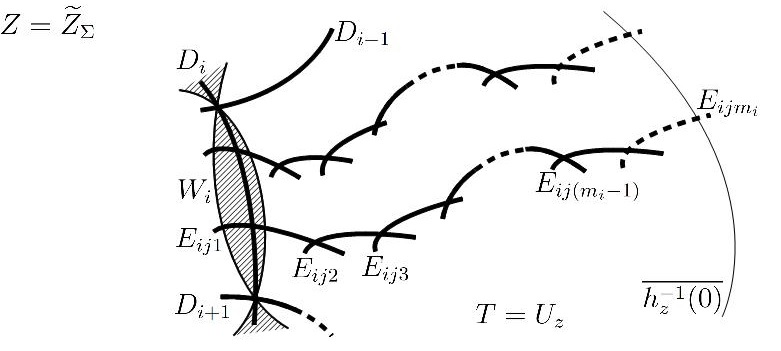}

Figure 5.
\end{center}


\begin{theorem}\label{2-dim-3} 
Assume that $c=(c_1, c_2) \in \CC^2$ is 
nonresonant, $\beta_i \notin \ZZ$ ($1 \leq i \leq l$) and 
$m_{i+1} \beta_i- m_i \beta_{i+1} 
\notin \ZZ$ ($1 \leq i \leq l-1$). 
Then the natural morphisms 
\begin{equation}
\Theta_i : H_2^{\rd} (W_i^{\circ}) \longrightarrow 
H_2^{\rd} (T^{\an}) \qquad (1 \leq i \leq l) 
\end{equation}
are injective and induce an isomorphism 
\begin{equation}
\Theta : \oplus_{i=1}^l H_2^{\rd} (W_i^{\circ}) \simto 
H_2^{\rd} (T^{\an}). 
\end{equation}
In particular, the cycles 
$\gamma_{ijk}:= \Theta_i ( \delta_{ijk}) 
\in H_2^{\rd}(T^{\an})$ 
($1 \leq i \leq l$, $1 \leq j \leq v_i$, 
$1 \leq k \leq m_i$) form a basis of 
the vector space 
 $H_2^{\rd}(T^{\an})=H_2 
(T^{\an} \cup Q, Q; \iota_* \LL )$ over $\CC$. 
\end{theorem} 

\begin{proof}
By the repeated use of 
Lemma \ref{2-dim-1} (i) and homotopy, 
we can find a small neighborhood $\wht{W_i}$ 
of $T_i^{\an} \cup \cup_{j=1}^{v_i} 
(E_{ij}^{\an} \setminus E_{ijm_i}^{\an})$ 
in $Z^{\an}$ containing $W_i$ such that for 
its open subset 
$\wht{W_i^{\circ}}=\wht{W_i} \cap 
T^{\an} \subset T^{\an}$ the natural morphism 
\begin{equation}
H_p^{\rd} (W_i^{\circ}) \longrightarrow 
H_p^{\rd} (\wht{W_i^{\circ}} )
\end{equation}
is an isomorphism for any $p \in \ZZ$. 
By our assumption 
$(c_1, c_2) \notin \ZZ^2$ we have 
the vanishing of the usual twisted 
homology group $H_p(T^{\an}; \LL )$ for 
any $p \in \ZZ$. Similarly we obtain 
the vanishings of 
$H_p( \wht{W_i^{\circ}} ; \LL )$ etc. 
Then the assertion can be proved by patching 
these results with the help 
of Lemma \ref{2-dim-1} (i) 
and the Mayer-Vietoris exact sequences for 
the relative twisted homology groups 
in the proof of Theorem \ref{Main}. 
\end{proof} 
Let $\Gamma_i \prec \Delta$ be 
the supporting face of 
$\rho_i$ in $\Delta$. Then 
the lattice length of $\Gamma_i$ is equal to 
$v_i \geq 0$ and we have the equality 
$\sum_{i=1}^{l} (v_i \times m_i)= 
\Vol_{\ZZ}(\Delta )$ as expected from 
the result of Theorem \ref{2-dim-3}.

\section{Higher-dimensional cases}\label{sec:7}

In this section, we shall extend the construction of 
the basis of the rapid decay homology group
$H_n (T^{\an} \cup Q, Q; \iota_* \LL )$ for $n=2$ 
in Section \ref{sec:6} to higher-dimensional cases.

\subsection{Some results on twisted Morse theory}

Let $T_0=(\CC^*)^k_x$ be a $k$-dimensional algebraic 
torus and $h_0(x)$ a Laurent polynomial 
on it whose Newton polytope $\Delta_0=NP(h_0) 
\subset \RR^k$ 
is $k$-dimensional. We assume that $h_0$ is 
non-degenerate in the sense of 
Kouchnirenko \cite{Kushnirenko}. Namely 
we impose the condition in Definition \ref{AND} 
for any face $\Gamma \prec \Delta_0$ of 
$\Delta_0$. 

\begin{proposition}\label{H-dim-1} 
In the situation as above, for generic 
$a=(a_1, \ldots, a_k) \in \CC^k$ the (possibly 
multi-valued) function $g(a,x)= h_0(x)x^{-a}$ 
on $T_0$ has exactly $\Vol_{\ZZ}(\Delta_0)$ 
critical points in $T_0$ and all of them 
are non-degenerate 
(i.e. of Morse type) and contained in 
$T_0 \setminus \{ h_0=0 \} = \{ x \in T_0 \ | \ 
g(a,x) \not= 0 \}$. 
\end{proposition}

\begin{proof} 
We follow the argument in \cite[page 10]{Esterov}. 
For $1 \leq i \leq k$ set $\partial_i= 
\partial_{x_i}$. Then for $x \in T_0$ we have 
\begin{equation}\label{Eqiva} 
\partial_i g(a,x)=0 \quad \Longleftrightarrow \quad 
 x_i \partial_i h_0(x)-a_i h_0(x)=0. 
\end{equation}
Since the hypersurface $\{ h_0=0 \} \subset T_0$ 
is smooth, by \eqref{Eqiva} 
all the critical points of the function 
$g(a, *)$ are contained in $T_0 \setminus \{ h_0=0 \}$. 
Set $f_i(a, x)= x_i \partial_i h_0(x)-a_i h_0(x)$. 
Then by Bernstein's theorem we have 
\begin{equation}
\sharp 
 \{ x \in T_0 \ | \ f_i(a,x)=0 \quad (1 \leq i \leq k) \} 
= \Vol_{\ZZ} (\Delta_0) 
\end{equation}
if the Newton polytopes $NP(f_i(a,*))$ of 
the Laurent polynomials 
$f_i(a,*)$ ($1 \leq i \leq k$) are equal to 
$\Delta_0$ and the subvariety 
$K= \{ x \in T_0 \ | \ f_i(a,x)=0 
\quad (1 \leq i \leq k) \}$ 
of $T_0$ is a non-degenerate complete intersection 
(see \cite{Oka}). 
From now on, we will show that these two conditions are 
satisfied for generic $a \in \CC^k$. First of all, 
it is clear that 
$NP(f_i(a,*))= \Delta_0$ ($1 \leq i \leq k$) 
for generic $a \in \CC^k$. 
Note that $x \in T_0$ is in 
$K$ if and only if 
$x \in T_0 \setminus \{ h_0=0 \}$ and 
\begin{equation}
\frac{x_i \partial_i h_0(x)}{h_0(x)}=a_i \qquad 
(1 \leq i \leq k), 
\end{equation}
that is, $x \in T_0 \setminus \{ h_0=0 \}$ is 
sent to the point $a \in \CC^k$ by the map 
$T_0 \setminus \{ h_0=0 \} \longrightarrow \CC^k$ 
defined by 
\begin{equation}
x \longmapsto \left( 
\frac{x_1 \partial_1 h_0(x)}{h_0(x)}, 
\ldots \ldots , 
\frac{x_k \partial_k h_0(x)}{h_0(x)} 
\right) . 
\end{equation}
By the Bertini-Sard theorem generic $a \in \CC^k$ 
are regular values of this map. If $a \in \CC^k$ 
is such a point, we can 
easily show that 
${\rm det} \{ \partial_j f_i(a,x) \}_{j,i=1}^k  
\not= 0$ for any $x \in 
K \subset T_0 \setminus \{ h_0=0 \}$. 
Now let $\Gamma \precneqq \Delta_0$ 
be a proper face of $\Delta_0$. Then 
for generic $a \in \CC^k$ we have 
\begin{equation}
\{ x \in T_0 \ | \ f_i(a,x)^{\Gamma}=0 \quad 
(1 \leq i \leq k) \} = \emptyset. 
\end{equation}
Indeed, let us assume the converse. Then the first 
projection from the variety 
\begin{equation}
\{ (a,x) \in \CC^k \times T_0 \ | \ 
 f_i(a,x)^{\Gamma}=0 \quad (1 \leq i \leq k) \} 
\subset \CC^k \times T_0 
\end{equation}
to $\CC^k$ is dominant. Moreover by $\Gamma \not= 
\Delta_0$ this variety is quasi-homogeneous with 
respect to the second variables $x=(x_1, \ldots, x_k)$. 
In particular, its dimension is greater than $k$. 
Then by considering the second 
projection from it to $T_0$, 
we find that there exist 
 $x \in T_0$ and $a \not= a^{\prime} \in 
\CC^k$ such that $f_i(a,x)^{\Gamma}= 
f_i(a^{\prime}, x)^{\Gamma}=0$ for any 
$1 \leq i \leq k$. Namely we have 
\begin{equation}\label{E1} 
 x_i \partial_i h_0^{\Gamma}(x)-a_i h_0^{\Gamma}(x) 
=0 \qquad (1 \leq i \leq k)
\end{equation}
and 
\begin{equation}\label{E2} 
 x_i \partial_i h_0^{\Gamma}(x)
-a_i^{\prime} h_0^{\Gamma}(x) 
=0 \qquad (1 \leq i \leq k). 
\end{equation}
Comparing \eqref{E1} with \eqref{E2} for 
$1 \leq i \leq k$ such that $a_i \not= 
a_i^{\prime}$, we obtain $h_0^{\Gamma}(x)=
 \partial_1 h_0^{\Gamma}(x)= \cdots 
= \partial_k h_0^{\Gamma}(x)=0$ for 
the point $x \in T_0$. This contradicts our 
assumption that the Laurent polynomial 
$h_0$ is non-degenerate. We thus proved that 
the subvariety $K= \{ x \in T_0 \ | \ f_i(a,x)=0 
\quad (1 \leq i \leq k) \}$ 
of $T_0$ is a non-degenerate complete 
intersection and its cardinality is 
$\Vol_{\ZZ}(\Delta_0)$ for generic $a \in \CC^k$. 
Let us fix such a point $a \in \CC^k$. 
Recall that $K= 
\{ x \in T_0 \ | \ f_i(a,x)=0 \quad 
(1 \leq i \leq k) \}$ is the set of the critical 
points of the function $g(a,*)$ in 
$T_0 \setminus \{ h_0=0 \}$. At such a critical 
point $x \in T_0 \setminus \{ h_0=0 \}$ we have 
\begin{equation}
\frac{\partial^2 g}{\partial x_j \partial x_i} 
(a,x)= \partial_j \left\{ f_i(a,x) \frac{x^{-a}}{x_i} 
\right\} 
= \partial_j f_i(a,x) \cdot 
 \frac{x^{-a}}{x_i}. 
\end{equation}
Since ${\rm det} \{ \partial_j f_i(a,x) 
\}_{j,i=1}^k \not= 0$, 
we obtain  
\begin{equation}
{\rm det} \left\{ \frac{\partial^2 g}{\partial x_j 
\partial x_i} 
(a,x) \right\}_{j,i=1}^k \not= 0. 
\end{equation}
Namely all the critical points of the function $g(a,*)$ 
are non-degenerate. 
\end{proof}

Let $h_0$ be as above and $\LL_0$ a local system of rank 
one on $T_0^{\an}$ defined by 
\begin{equation}
\LL_0=\CC_{T_0^{\an}}x_1^{\beta_1} \cdots x_k^{\beta_k} 
\qquad ( \beta =( \beta_1, \ldots, \beta_k) \in \CC^k). 
\end{equation}
From now on, we will calculate the twisted homology 
groups $H_p(T_0^{\an} \setminus \{ h_0=0 \}^{\an} ; 
\LL_0)$ $(p \in \ZZ)$ by using our twisted Morse 
theory. Taking a sufficiently generic 
$a \in \Int (\Delta_0) \subset \RR^k$ 
we set $h_1(x)=h_0(x)x^{-a}$. Then by Proposition 
\ref{H-dim-1} the real-valued function 
$f:= |h_1|^{-2}: 
T_0^{\an} \setminus  \{ h_0=0 \}^{\an} 
\longrightarrow \RR$ has only $\Vol_{\ZZ}(\Delta_0)$ 
non-degenerate (Morse) critical points in 
$M:= T_0^{\an} \setminus  \{ h_0=0 \}^{\an}$. Moreover 
we can easily verify that the index of such a 
critical point is always $k$. 
For $t \in \RR_{>0}$ we define an open subset 
$M_t \subset M$ of $M$ by 
\begin{equation}
M_t = \{ x \in M \ | \ f(x)=|h_1|^{-2}(x) <t \}. 
\end{equation}
Then we have the following result. 

\begin{proposition}\label{VAN}  
For generic 
$\beta =( \beta_1, \ldots, \beta_k) \in \CC^k$ 
and $0< \e \ll 1$ we have 
\begin{equation}
H_p (M_{\e}; \LL_0) \simeq 0
\end{equation}
for any $p \in \ZZ$. 
\end{proposition} 

\begin{proof} 
Let $\Sigma^{\prime}$ be a smooth subdivision of 
the dual fan of $\Delta_0$ and $Z_{\Sigma^{\prime}}$ 
the (smooth) toric variety associated to it. Then 
the divisor at infinity $D^{\prime}= 
Z_{\Sigma^{\prime}} \setminus T_0$ is 
normal crossing. By the non-degeneracy of 
$h_0$ the divisor $\overline{h_0^{-1}(0)} \cup 
D^{\prime}$ is also normal crossing 
in a neighborhood of $D^{\prime}$. 
Moreover by the 
condition $a \in \Int (\Delta_0)$ 
the neighborhood 
$M_{\e} \sqcup 
(D^{\prime} \setminus \overline{h_0^{-1}(0)})$ of 
$D^{\prime} \setminus \overline{h_0^{-1}(0)}$ 
retracts to $D^{\prime} \setminus 
\overline{h_0^{-1}(0)}$ 
as $\e \longrightarrow +0$. Since for 
generic $\beta \in \CC^k$ the local system 
$\LL_0$ has a non-trivial monodromy around 
each irreducible component of $D^{\prime}$, 
the assertion follows. 
\end{proof} 
By this proposition we can apply the argument in 
the proof of Theorem \ref{BRH} to the Morse 
function $f= |h_1|^{-2}: M=
 T_0^{\an} \setminus  \{ h_0=0 \}^{\an} 
\longrightarrow \RR$ and obtain the following 
theorem. 

\begin{theorem}\label{TOD} 
For generic 
$\beta =( \beta_1, \ldots, \beta_k) \in \CC^k$ 
we have 
\begin{equation}
\dim 
 H_p (T_0^{\an} \setminus  \{ h_0=0 \}^{\an} ; \LL_0 )
=\begin{cases}
 \Vol_{\ZZ}(\Delta_0)  & (p=k), \\
\ 0 & (p \not= k) 
\end{cases}
\end{equation}
and there exists a basis of 
$H_k (T_0^{\an} \setminus  \{ h_0=0 \}^{\an} ; \LL_0 )$ 
indexed by the $\Vol_{\ZZ}(\Delta_0)$ 
non-degenerate (Morse) critical points 
of the (possibly multi-valued) 
function $h_1(x)=h_0(x)x^{-a}$ 
in $T_0 \setminus  \{ h_0=0 \}$. 
\end{theorem} 

Note that Theorem \ref{TOD} partially solves 
the famous problem in the paper Gelfand-Kapranov-Zelevinsky 
\cite{G-K-Z-2} of constructing a basis of 
the twisted homology group in their integral 
representation of $A$-hypergeometric 
functions. Indeed Theorem \ref{TOD} holds even if 
we replace the local system $\LL_0$ with the one 
\begin{equation}
\LL_1=
\CC_{T_0^{\an} \setminus \{ h_0=0 \}^{\an}}  
h_0(x)^{\alpha} x_1^{\beta_1} \cdots x_k^{\beta_k} 
\qquad ( 
\alpha \in \CC, \ \beta =( 
\beta_1, \ldots, \beta_k) \in \CC^k) 
\end{equation}
on $T_0^{\an} \setminus \{ h_0=0 \}^{\an}$.

\subsection{A construction of the basis in 
the higher-dimensional case}

Now we consider the situation 
in Sections \ref{sec:4} and \ref{sec:5}.  
We inherit the notations there. 
We fix a point $z \in \Omega$ and 
define $Q \subset \tilde{D} \subset \tilde{Z}$ etc. 
in the real oriented blow-up $\pi : \tilde{Z} 
\longrightarrow Z^{\an}$ of $Z^{\an}=
( \tl{Z_{\Sigma}} )^{\an}$ as in the proof 
of Theorem \ref{Main}. For the local system 
$\LL =\CC_{T^{\an}} x_1^{c_1-1} \cdots 
x_n^{c_n-1}$ on $T^{\an}$ we shall construct 
a basis of the rapid decay homology 
group $H_n^{\rd}(T^{\an}):=H_n 
(T^{\an} \cup Q, Q; \iota_* \LL )$. As 
in Section \ref{sec:6}, for an open subset 
$W$ of $T^{\an}$ and $p \in \ZZ$ we set 
\begin{equation}
H_p^{\rd}(W):=H_p(W \cup Q, Q; \iota_* \LL )
\end{equation}
for short. Recall that $\rho_1, \ldots, \rho_{l}$ 
are the rays in the smooth 
fan $\Sigma$ which correspond to the relevant 
divisors $D_1, \ldots, D_{l}$ 
in $Z= \tl{Z_{\Sigma}}$. 
By the primitive vector $\kappa_i \in \rho_i 
\cap (\ZZ^n \setminus \{ 0 \} )$ on 
$\rho_i$ the order $m_i>0$ of the pole 
of $h_z(x)=\sum_{j=1}^N z_j x^{a(j)}$ 
along $D_i$ is explicitly 
described by 
\begin{equation}
m_i= - \min_{a \in \Delta} 
\langle \kappa_i, a \rangle . 
\end{equation} 
By the non-degeneracy 
of $h_z$ the complex hypersurface $\overline{h_z^{-1}(0)} 
\subset Z_{\Sigma}$ intersects each relevant 
divisor $D_i$ transversally. Let 
$T_i \simeq (\CC^*)^{n-1} \subset D_i$ be 
the $T$-orbit in 
$Z_{\Sigma}$ which corresponds to $\rho_i$ 
and denote by the same letter $T_i$ its 
strict transform in the blow-up 
$Z=\tl{Z_{\Sigma}}$. 
Recall also that for $1 \leq i \leq l$ 
the Euler characteristic of 
the hypersurface $\{ h_z^{\Gamma_i}=0 \} 
=T_i \cap \overline{h_z^{-1}(0)} $ 
in $T_i$ is equal to 
$(-1)^n v_i=(-1)^n \Vol_{\ZZ}( \Gamma_i)$. Let 
$y=(y_1, \ldots, y_n)$ be the coordinates 
on an affine chart 
$U_i \simeq \CC^n \subset Z_{\Sigma}$ 
of $Z_{\Sigma}$ containing 
$T_i$ such that $T_i= \{ y_n=0 \}$ and 
$\LL \simeq \CC_{T^{\an}} y_1^{\beta_1} 
\cdots y_{n-1}^{\beta_{n-1}}y_n^{\beta_n}$. 
Define a local system $\LL_i$ on $T_i^{\an}$ by 
$\LL_i=\CC_{T_i^{\an}} y_1^{\beta_1} 
\cdots y_{n-1}^{\beta_{n-1}}$. 

\begin{proposition}\label{TODI} 
If $(\beta_1, \ldots, \beta_{n-1}) \in \CC^{n-1}$ 
is generic, we have 
\begin{equation}
\dim 
 H_p (T_i^{\an} \setminus  
\{ h_z^{\Gamma_i}=0 \}^{\an} ; \LL_i )
=\begin{cases}
v_i  & (p=n-1), \\
\ 0 & (p \not= n-1) 
\end{cases}
\end{equation}
and can construct a basis of 
$H_{n-1} (T_i^{\an} \setminus  
\{ h_z^{\Gamma_i}=0 \}^{\an} ; \LL_i )$ 
by the twisted Morse theory. 
\end{proposition} 

\begin{proof} 
If $\dim \Gamma_i =n-1$ ($\Longleftrightarrow 
v_i>0$), the assertion follows immediately 
from Theorem \ref{TOD}. If $\dim \Gamma_i <n-1$ 
($\Longleftrightarrow 
v_i=0$), we have $T_i^{\an} \setminus  
\{ h_z^{\Gamma_i}=0 \}^{\an} \simeq 
\CC^* \times W$ for an open subset 
$W$ of $(\CC^*)^{n-2}$. Hence 
for generic $(\beta_1, \ldots, 
\beta_{n-1}) \in \CC^{n-1}$ there 
exists an isomorphism  
$H_p (T_i^{\an} \setminus  
\{ h_z^{\Gamma_i}=0 \}^{\an} ; \LL_i )
 \simeq 0$ for any $p \in \ZZ$. 
\end{proof} 

Similarly we can prove also 
the following proposition. 

\begin{proposition}\label{TAK} 
For each generic parameter 
vector $c \in \CC^n$ 
and $1 \leq i \leq l$ there exists a 
sufficiently small tubular neighborhood $W_i$ 
of $T_i^{\an}$ in $Z^{\an}$ such that for 
its open subset $W_i^{\circ}=W_i \cap 
T^{\an} \subset T^{\an}$ we have 
\begin{equation}
\dim H_p^{\rd} (W_i^{\circ})
=\begin{cases}
v_i \times m_i & (p=n), \\
\ 0 & (p \not= n). 
\end{cases}
\end{equation}
and can construct a basis 
$\delta_{ijk} \in H_n^{\rd} (W_i^{\circ})$ 
($1 \leq j \leq v_i$, $1 \leq k \leq m_i$) 
of $H_n^{\rd} (W_i^{\circ})
=H_n(W_i^{\circ} \cup Q, 
Q; \iota_* \LL )$ 
by the twisted Morse theory. 
\end{proposition} 

\begin{proof} 
Let $f_i: T_i^{\an} \setminus  
\{ h_z^{\Gamma_i}=0 \}^{\an} \longrightarrow 
\RR$ be the function on $T_i^{\an} \setminus  
\{ h_z^{\Gamma_i}=0 \}^{\an}$ defined by 
$f_i(x)=| h_z^{\Gamma_i}(x)x^{-a} |^{-2}$ 
for a sufficiently generic $a \in \Int 
( \Gamma_i) \subset \RR^{n-1}$. Then by 
Proposition \ref{H-dim-1} the function 
$f_i$ has only $v_i$ non-degenerate 
(Morse) critical points in 
$T_i^{\an} \setminus  
\{ h_z^{\Gamma_i}=0 \}^{\an}$. By the 
product decomposition $U_i^{\an} \simeq 
\CC^n_y = \CC^{n-1}_{y_1, \ldots, y_{n-1}} 
\times \CC_{y_n}$ we consider $f_i$ also 
as a function on the open subset 
\begin{equation}
U_i^{\circ}=
\left( 
T_i^{\an} \setminus  
\{ h_z^{\Gamma_i}=0 \}^{\an} 
\right) \times \CC_{y_n} \subset U_i^{\an} 
\end{equation}
of $U_i^{\an}$. For $t \in \RR_{>0}$ we set 
\begin{equation}
W_{i,t}^{\circ} = \{ y \in U_i^{\circ} \cap 
W_i^{\circ} \ | \ f_i(y)< t \}. 
\end{equation}
Then it follows from Lemma \ref{2-dim-1} (i) 
that by shrinking $W_i$ and taking 
large enough $t_0 \gg 1$ we obtain isomorphisms 
\begin{equation}\label{EEE1} 
H_p^{\rd}(W_{i,t_0}^{\circ}) \simto 
H_p^{\rd}(W_{i}^{\circ}) \quad (p \in \ZZ ). 
\end{equation}
In the same way as the proof of Proposition 
\ref{VAN}, for generic $c \in \CC^n$ 
and sufficiently small 
$0< \e \ll 1$ we can show that 
$H_p^{\rd}(W_{i,\e}^{\circ}) \simeq 0$ 
$(p \in \ZZ )$. 
Let $0 < t_1 < t_2 < \cdots < t_r < 
+ \infty$ be the critical values of 
$f_i: T_i^{\an} \setminus  
\{ h_z^{\Gamma_i}=0 \}^{\an} \longrightarrow 
\RR$. We may assume that $t_r <t_0$. 
For $1 \leq j \leq r$ 
let $\alpha (1), \alpha (2), 
\ldots, \alpha ( n_j) \in 
T_i^{\an} \setminus  
\{ h_z^{\Gamma_i}=0 \}^{\an}$ be the 
critical points of $f_i$ such that 
$f_i (\alpha (q))=t_j$. 
Let $S_q \subset T_i^{\an} \setminus  
\{ h_z^{\Gamma_i}=0 \}^{\an}$ be the stable 
manifold of the gradient flow of 
$f_i$ passing through 
its critical point $\alpha (q) \in 
T_i^{\an} \setminus  
\{ h_z^{\Gamma_i}=0 \}^{\an}$ and 
$B_{\e}^* \subset \CC^*_{y_n}$ the punctured 
disk $\{ y_n \in \CC \ | \ 0< | y_n | < \e \}$ 
for sufficiently small 
$0 < \varepsilon \ll 1$. 
We assume here that $S_q$ are homeomorphic 
to the $(n-1)$-dimensional disk. 
Then by Lemma \ref{ECL} 
and the K\"{u}nneth formula we have 
\begin{equation} 
H_p( (\overline{S_q} \times B_{\e}^*) \cup Q, 
(\partial S_q \times B_{\e}^*) \cup Q ; 
\iota_* \LL ) \simeq 0 \ \ \ \ 
( p \not= n) 
\end{equation}
and the dimension of the vector space 
$H_n( (\overline{S_q} \times B_{\e}^*) \cup Q, 
(\partial S_q \times B_{\e}^*) \cup Q ; 
\iota_* \LL )$ is $m_i$. 
This implies that 
for any $1 \leq j \leq r$ and 
generic $c \in \CC^n$ there exists a 
short exact sequence 
\begin{eqnarray}
0 \longrightarrow  
H_n^{\rd}( W_{i,t_j- \e }^{\circ} ) 
\longrightarrow 
H_n^{\rd}( W_{i,t_j+ \e }^{\circ} ) 
\longrightarrow 
\\
\oplus_{q=1}^{n_j} 
H_n( (\overline{S_q} \times B_{\e}^*) \cup Q, 
(\partial S_q \times B_{\e}^*) \cup Q ; 
\iota_* \LL ) \longrightarrow 0. 
\end{eqnarray}
Hence by \eqref{EEE1} we get 
\begin{equation}
\dim H_p^{\rd} (W_i^{\circ})
=\begin{cases}
v_i \times m_i & (p=n), \\
\ 0 & (p \not= n). 
\end{cases}
\end{equation}
Moreover by Lemma \ref{ECL} we can construct a 
natural basis of the vector space 
$H_n( (\overline{S_q} \times B_{\e}^*) \cup Q, 
(\partial S_q \times B_{\e}^*) \cup Q ; 
\iota_* \LL ) \simeq \CC^{m_i}$. 
Lifting these bases to 
$H_n^{\rd}(W_{i,t_0}^{\circ}) \simeq 
H_n^{\rd}(W_{i}^{\circ})$ with the help of 
the above short exact sequences we obtain 
the one $\delta_{ijk} \in H_n^{\rd} (W_i^{\circ})$ 
($1 \leq j \leq v_i$, $1 \leq k \leq m_i$) 
of $H_n^{\rd}(W_{i}^{\circ})$. 
\end{proof} 

Since for generic $c \in \CC^n$ we have 
the vanishings of the usual twisted 
homology groups $H_p(T^{\an}; \LL )$ and 
$H_p( W_i^{\circ} ; \LL )$ etc., 
by Mayer-Vietoris exact sequences for 
relative twisted homology 
groups and Proposition 
\ref{TAK} we obtain the following theorem. 

\begin{theorem}\label{H-dim} 
For generic nonresonant 
parameter vectors $c \in \CC^n$ 
the natural morphisms 
\begin{equation}
\Theta_i : H_n^{\rd} (W_i^{\circ}) \longrightarrow 
H_n^{\rd} (T^{\an}) \qquad (1 \leq i \leq l) 
\end{equation}
are injective and induce an isomorphism 
\begin{equation}
\Theta : \oplus_{i=1}^l H_n^{\rd} (W_i^{\circ}) \simto 
H_n^{\rd} (T^{\an}). 
\end{equation}
In particular, the cycles 
$\gamma_{ijk}:= \Theta_i ( \delta_{ijk}) 
\in H_n^{\rd}(T^{\an})$ 
($1 \leq i \leq l$, $1 \leq j \leq v_i$, 
$1 \leq k \leq m_i$) form a basis of 
the vector space 
 $H_n^{\rd}(T^{\an})=H_n 
(T^{\an} \cup Q, Q; \iota_* \LL )$ over $\CC$. 
\end{theorem} 

\section{Addendum}\label{sec:add}

In previous sections we obtained the integral representation  of 
confluent $A$-hypergeometric functions by using 
Hien's theory of rapid decay homologies. However we assumed 
there that the parameter $c \in \CC^n$ is 
non-resonant. Since the resonant case $c=(1,1,\ldots , 1)$ 
is important also in mirror symmetry, it is desirable 
to know what happens for resonant parameters. 
In Definition 3.4 of \cite{SW1} Schulze and Walther defined the set 
$\text{sRes}(A) \subset \CC$ of strongly-resonant parameters 
(see also Section 2.2 of \cite{R-W}). 
By their result and Proposition 3.11 of \cite{R-S}, 
we can easily check that the main results in 
Sections 4 and 5 of this paper 
hold true also for non-strongly-resonant parameters 
$c \notin \text{sRes}(A)$. 
Moreover, if $A$ is saturated 
i.e. $\ZZ^n \cap \RR_+ A= \NN A \cap \RR_+ A$, then 
we have $c=(1,1,\ldots , 1) \notin \text{sRes}(A)$. Note 
that under our assumption $\ZZ A = \ZZ^n$  
this condition is always satisfied if the polytope 
$\Delta_A = \text{conv} ( \{0\} \cup A) \subset \RR^n$ 
contains the origin $0 \in \RR^n$ in its 
interior. This implies in particular that 
for such $A$ the asymptotic expansions 
at infinity of confluent $A$-hypergeometric 
functions proved in Sections 5 of this paper holds 
for the resonant parameter $c=(1,1,\ldots , 1)$. 

\bigskip
\noindent{\bf Acknowledgement:} 
We thank Professor Thomas Reichelt for useful discussions 
for this addendum.

\end{document}